\documentclass[11pt]{article}
\usepackage{amsmath}
\usepackage{amsthm}
\usepackage{amsfonts}
\usepackage{amssymb}
\usepackage{enumerate}
\usepackage{dsfont}
\usepackage{graphicx}
\usepackage{epstopdf}
\usepackage{geometry}
\usepackage{stmaryrd}
\usepackage[colorlinks=true,linkcolor=blue,urlcolor=blue,citecolor=red, hyperfigures=false]{hyperref}

\let\inf\relax \DeclareMathOperator*\inf{\vphantom{p}inf}
\let\max\relax \DeclareMathOperator*\max{\vphantom{p}max}
\let\min\relax \DeclareMathOperator*\min{\vphantom{p}min}
\let\argmax\relax \DeclareMathOperator*\argmax{\vphantom{p}argmax}
\let\argmin\relax \DeclareMathOperator*\argmin{\vphantom{p}argmin}

\newcommand{\be}{\begin{equation}}
\newcommand{\ee}{\end{equation}}

\numberwithin{equation}{section}
\numberwithin{figure}{section}
\newtheorem{theorem}{Theorem}[section]
\theoremstyle{plain}

\newtheorem{definition}[theorem]{Definition}

\newtheorem{lemma}[theorem]{Lemma}

\newtheorem{proposition}[theorem]{Proposition}
\newtheorem{remark}[theorem]{Remark}

\numberwithin{figure}{section}

\newcommand{\RR}{{\mathbb{R}}}
\newcommand{\FF}{{\mathbb{F}}}
\newcommand{\DD}{{\mathbb{D}}}
\newcommand{\EE}{{\mathbb{E}}}

\newcommand{\PP}{{\mathbb{P}}}

\newcommand{\CF}{{\mathcal{F}}}

\newcommand{\CV}{{\mathcal{V}}}

\newcommand{\CU}{{\mathcal{U}}}
\newcommand{\CO}{{\mathcal{O}}}

\newcommand{\CM}{{\mathcal{M}}}
\newcommand{\CA}{{\mathcal{A}}}

\newcommand{\CS}{{\mathcal{S}}}

\newcommand{\CT}{{\mathcal{T}}}
\newcommand{\CB}{{\mathcal{B}}}

\newcommand{\indic}{{\mathds{1}}}
\newcommand{\comment}[1]{}

\def\p{\vskip4truept \noindent}

\newcommand{\tr}{{\!\!~^\top\!\!}}
\newcommand{\Leb}{{\rm Leb}}
\newcommand{\tq}{{\, | \,}}
\title{Continuous-time Markov games with asymmetric information.}
\author{Fabien Gensbittel\thanks{Toulouse School of Economics, University Toulouse 1 Capitole, \newline
 Manufacture des Tabacs, MF213, 21, All\'{e}e de Brienne 31015 Toulouse Cedex 6. \newline E-mails:\href{mailto:fabien.gensbittel@tse-fr.eu}{fabien.gensbittel@tse-fr.eu}}}

\begin{document}
\maketitle

\begin{abstract}
We study a two-player zero-sum stochastic differential game with asymmetric information where the payoff depends on a controlled continuous-time Markov chain $X$ with finite state space  which is only observed by player 1. 
This model was already studied in Cardaliaguet et al \cite{cardaetal} through an approximating sequence of discrete-time games. Our first contribution is the proof of the existence of the value in the continuous-time model based on duality techniques. This value is shown to be the unique solution of the same Hamilton-Jacobi equation with convexity constraints which characterized  the limit value obtained in \cite{cardaetal}. Our second main contribution is to provide a simpler equivalent formulation for this Hamilton-Jacobi equation using directional derivatives and exposed points, which we think is interesting for its own sake as the associated comparison principle has a very simple proof which avoids all the technical machinery of viscosity solutions.
\end{abstract}


\p
\p
\section{Introduction} 

The present work contributes to the literature on zero-sum differential games with incomplete information, and is more precisely related to the model of differential games with asymmetric information developed in Cardaliaguet \cite{cardadiff} which
already led to various extensions and generalizations (see e.g. Cardaliaguet \cite{cardadouble}, Cardaliaguet and Rainer \cite{cardastochdiff}\cite{cardaexemple}\cite{carda12}, Gr\"{u}n \cite{grunstopping}\cite{grungirsanov}, Oliu-Barton \cite{Oliubarton}, Buckdahn, Quinquampoix, Rainer and Xu \cite{BQRW}, Jimenez, Quincampoix and Xu \cite{JQW}, Wu \cite{WU}, Jimenez and Quincampoix \cite{JQ}).

Most of the literature on zero-sum dynamic games with asymmetric information, including the above mentioned works, deals with models where the payoff-relevant parameters of the game that are partially unknown (say information parameters) do not evolve over time. Some recent works focus on models of dynamic games with asymmetric information and evolving information parameters. Discrete-time models were analyzed in Renault \cite{Renault2006}, Neyman \cite{neyman}, Gensbittel and Renault \cite{GensbittelRenault}; some continuous-time models were analyzed using an approximating sequence of discrete-time games in  Cardaliaguet, Rainer, Rosenberg and Vieille \cite{cardaetal}, Gensbittel \cite{gensbittel2013} and Gensbittel and Rainer \cite{obsbrown}, and a model of continuous-time stopping game was analyzed in Gensbittel and Gr\"{u}n \cite{GensbittelGrun}.

In this paper we consider a two player zero-sum stochastic differential game with asymmetric information. The payoff depends on some continuous time controlled Markov chain $(X_t)_{t \geq 0}$ with finite state space $K$, having a commonly known initial law $p$ and infinitesimal generator $R(u_t,v_t)_{t \geq 0}$ where $u_t$ and $v_t$ are respectively the controls of player $1$ and player $2$. We assume that $X$ is only observed by player 1 while the controls are publicly observed, so that the control $u_t$ depends on the trajectory of $X$ up to time $t$ while the control $v_t$ does not. The payoff of player 1 is given by
\[ \EE[ \int_0^\infty r e^{-rt} g(X_t,u_t,v_t)dt], \]
where $r$ is the discount factor and $g$ is a bounded payoff function. 
This model is therefore a continuous-time version of the model of discrete-time stochastic games with discounted payoffs where the state variable is only observed by player $1$ and actions are publicly observed. In particular, there is incomplete information about a stochastic process evolving over time.   
We prove that this game has a value $W(p)$ when players are allowed to use suitable mixed non-anticipative strategies and provide a variational characterization for $W$.

This model is was already studied by  Cardaliaguet, Rainer, Rosenberg and Vieille in \cite{cardaetal}. However, the analysis in \cite{cardaetal} was only done through an approximating sequence of discrete-time games in which the players play more and more frequently. Let us emphasize that no formal definition of the continuous-time game was given in \cite{cardaetal}.

In this work, we define the continuous-time game formally and prove the existence of the value $W(p)$ in the continuous-time model directly. We prove that $W$ is the unique solution of the same Hamilton-Jacobi equation with convexity constraints that was introduced in  \cite{cardaetal} to characterize the limit of the values of the discrete-time games along the approximating sequence.

Our second main contribution is to obtain an equivalent simpler formulation for this Hamilton-Jacobi equation which is reminiscent of the variational representation for the value of repeated games with asymmetric information given by Mertens and Zamir \cite{MZ} and actually inspired by the notion of dual solution initially proposed by Cardaliaguet in \cite{cardadiff} (see also Gensbittel and Gr\"{u}n \cite{GensbittelGrun} for a similar formulation in the context of stopping games). One of the main advantage of such a formulation is that the associated comparison principle has a very simple proof which avoids all the complex machinery of viscosity solutions.

The paper is structured as follows. In section \ref{section_model} we give a formal description of the model and state the main results. In section \ref{section_HJB}, we analyze the Hamilton-Jacobi equation with convexity constraints introduced in \cite{cardaetal} and provide an equivalent simpler formulation together with a simple proof of the associated comparison principle. In section \ref{section_DPP}, we prove that the game has a value which is the unique solution of the Hamilton-Jacobi equation analyzed in section \ref{section_HJB}.

\section{Model and main results}\label{section_model}

\subsection{Notation}

Let $K$ be a non-empty finite set which we identify with $\{1,...,|K|\}$ and $\Delta(K)=\{ p \in \RR^{K} \tq \forall k \in K, p_k \geq 0, \; \sum_{k \in K} p_k =1\}$ be the set of probabilities over $K$. We use the notation $\delta_k \in \Delta(K)$ for the Dirac mass at $k \in K$.

Let $\Omega=\DD([0,\infty), K)$ denote the set of c\`{a}dl\`{a}g (right-continuous with left limits) trajectories $\omega=(\omega(t))_{t \geq 0}$ taking values in $K$ ($K$ being endowed with the discrete topology). For all $t\geq 0$, the canonical process on $\Omega$ is defined by $X_t(\omega)=\omega(t)$ and $\FF^X=(\CF^X_t)_{t \geq 0}$ denotes the canonical filtration, i.e. $\CF^X_t=\sigma(X_s, 0\leq s\leq t)$. 
We define $\CF=\CF_\infty$.
For all $t\geq 0$, let $\Omega_t=\DD([0,t], K)$, which is endowed with the $\sigma$-algebra generated by the projections $\omega_t \in \Omega_t \rightarrow \omega_t(s)$ for all $s\in [0,t]$.

Let $U$ and $V$ be non-empty Polish spaces which represent the sets of controls for player $1$ and $2$ respectively. 
Let $\CU$ (resp. $\CU_t$ for all $t \geq 0$) denote the set of Borel-measurable maps from $[0,\infty)$ (resp. $[0,t]$) to $U$, endowed with the topology of convergence in Lebesgue measure. 
The sets $\CV$ and $\CV_t$ for all $t \geq 0$ of $V$-valued maps are defined similarly.
Note that the above definition implies that $\CU_{0}$ and $\CV_{0}$ are endowed with the trivial $\sigma$-algebra. 

Let $\CM$ denote the set of $K\times K$ matrices $M=(M_{i,j})_{i,j \in K}$ of transition rates, i.e. for all $i,j \in K$ such that $i\neq j$, $M_{i,j} \geq 0$ and for all $i\in K$, $M_{i,i}=-\sum_{j \neq i} M_{i,j}$.

Let $r>0$ be a positive discount factor, $g: K\times U \times V \rightarrow [0,1]$ be a measurable payoff function and $R: U\times V \rightarrow \CM$ a bounded measurable intensity function.

In the sequel, all the topological spaces $E$ are endowed with their Borel $\sigma$-algebra denoted $\CB(E)$, all the products are endowed with the product $\sigma$-algebra.

For any function $f: I \rightarrow E$ and $J\subset I$, $f|_J$ denotes the restriction of $f$ to $J$.

In order to define the game, we first need to recall what is a controlled Markov chain $X$. At first, the term Markov chain is abusively used here, as for controlled diffusions, since the processes we consider are not Markovian, and an alternative (less ambiguous) denomination could be jump processes with controlled intensity.

Let $\PP$ be a probability measure on $(\Omega,\CF)$ and $\Lambda$ be an $\FF^X$-progressively measurable bounded process with values in $\CM$. For all $i,j \in K$ with $i\neq j$, let $N^{i,j}_t$ denotes the number of jumps of $X$ from state $i$ to state $j$ in the time-interval $(0,t]$. The process $X$ is said to have $\FF^X$-intensity $\Lambda$ if for all pairs $(i,j) \in K^2$ with $i\neq j$ 
the counting process $N_t^{i,j}$ has $\FF^X$-intensity $\indic_{X_t=i}\Lambda_{i,j}$, which means that for all non-negative $\FF$-predictable processes $Z$
\[ \EE[\int_0^\infty Z_s dN^{X,i,j}_s ] =\EE[\int_0^\infty Z_s \indic_{X_s=i}\Lambda_{i,j} ds ],\]
or equivalently that the process $N^{i,j}_t - \int_0^t \indic_{X_s=i}\Lambda_{i,j} ds$ is a $(\PP,\FF^X)$ martingale.

\subsection{Strategies}

In order to avoid all the technical considerations related to measurability, we work here with piecewise-constant controls that are left-continuous. Note however that the results can easily be extended to larger families of controls and strategies.

\begin{definition}\label{def:grid}
We call $T=\{t_i,i\geq 0\}$ a grid if the sequence $(t_i)_{i\geq 0}$ is increasing, and satisfies $t_0=0$ and $\lim_{i \rightarrow \infty}t_i=+\infty$. We say that a grid $T'$ is finer than $T$ if $T' \subset T$.
\end{definition}

Let us define $\CS=\{([0,1]^n,\CB([0,1]^n),\Leb^{\otimes n}), n \geq 1 \}$, where $\Leb$ denotes the Lebesgue measure. $\CS$ is the family of probability spaces available to the players and is stable by products.

\begin{definition}
A \textbf{pure strategy for player 2} is a measurable map $\beta: \CU \rightarrow \CV$ such that there exists a grid $T=\{t_i, i \geq 0\}$ such that 
\[ \forall t \geq 0,\; \beta(u)_t = \sum_{i \geq 0}\indic_{(t_i,t_{i+1}]}(t) \beta^i(u|_{[0,t_i]}) ,\]
where for all $i\geq 0$, $\beta^i : \CU_{t_i} \rightarrow V$ is measurable. The set of pure strategies of player $2$ is denoted $\CT$.
\p
A \textbf{mixed strategy  for player 2} is a  pair $((M_\beta,\CA_\beta,\lambda_\beta),\beta)$ where $(M_\beta,\CA_\beta,\lambda_\beta)$ is a probability space in $\CS$ and $\beta: M_\beta \times \CU \rightarrow \CV$ is a measurable map such that there exists a grid $T=\{t_i, i \geq 0\}$ such that 
\[ \forall t \geq 0, \;\beta(\xi_\beta,u)_t = \sum_{i \geq 0}\indic_{(t_i,t_{i+1}]}(t) \beta^i(\xi_\beta,u|_{[0,t_i]}) ,\]
where for all $i\geq 0$, $\beta^i : M_\beta\times \CU_{t_i} \rightarrow V$ is measurable. The set of mixed strategies of player 2 is denoted $\widehat{\CT}$.
Note that for all $\xi_\beta\in M_\beta$, $\beta(\xi_\beta,.)$ is  a pure strategy with grid $T$ that will be denoted $\beta(\xi_\beta)$.
\p
A \textbf{pure strategy  for player 1} is a measurable map $\alpha: \Omega\times \CV \rightarrow \CU$  such that there exists a grid $T=\{t_i, i \geq 0\}$ with 
\[ \forall t \geq 0,\; \alpha(\omega,v)_t = \sum_{i \geq 0}\indic_{(t_i,t_{i+1}]}(t) \alpha^i(\omega|_{[0,t_i]},v|_{[0,t_i]}) ,\]
where for all $i\geq 0$, $\alpha^i : \Omega_{t_i}\times \CV_{t_i} \rightarrow U$ is measurable. The set of pure strategies of player $1$ is denoted $\Sigma$.
\p
A \textbf{mixed strategy for player 1} is a  pair $((M_\alpha,\CA_\alpha,\lambda_\alpha),\alpha)$ where $(M_\alpha,\CA_\alpha,\lambda_\alpha)$ is a probability space in $\CS$ and $\alpha: M_\alpha \times \Omega \times \CV \rightarrow \CU$ is a measurable map such that there exists a grid $T=\{t_i, i \geq 0\}$ such that 
\[ \forall t \geq 0, \;\alpha(\xi_\alpha,\omega,v)_t = \sum_{i \geq 0}\indic_{(t_i,t_{i+1}]}(t) \alpha^i(\xi_\alpha,\omega|_{[0,t_i]},v|_{[0,t_i]}) ,\]
where for all $i\geq 0$, $\alpha^i : M_\alpha \times \Omega_{t_i}\times  \CV_{t_i} \rightarrow U$ is measurable. The set of mixed strategies of player 1 is denoted $\widehat\Sigma$. 
Note that for all $\xi_\alpha \in M_\alpha$, $\alpha(\xi_\alpha,.)$ is a pure strategy with grid $T$ that will be denoted $\alpha(\xi_\alpha)$.
\end{definition}

We will simply write $\alpha$ (resp. $\beta$) instead of $((M_\alpha,\CA_\alpha,\lambda_\alpha),\alpha)$ (resp. $((M_\beta,\CA_\beta,\lambda_\beta),\beta)$) whenever there is no risk of confusion. 

We will identify pure strategies as particular mixed strategies in which the probability space is reduced to as single point.

\begin{remark}
According to the above definition, the value at time $0$ of all the controls induced by a strategy is fixed to be zero. However, the measurable maps appearing in the definition of the strategies do not depend on this value since the atoms of the Borel $\sigma$-algebra of $\CU_t$ and $\CV_t$ for $t \geq 0$ are equivalence classes of functions with respect to equality Lebesgue almost everywhere. Note also that with our definition of $\CU_{0}$ and $\CV_{0}$, the value of a control induced by strategy on the first interval $(t_0,t_1]$ of the grid does not depend on the control of his opponent.  
Moreover, the reader may check that the probabilities on $\Omega$ we consider later do not depend on the value at zero of the controls, and this will be implicitly used when proving dynamic programming inequalities since we will consider continuation controls for which formally the value at time zero is not zero.   
\end{remark}
The next lemma contains an obvious but useful remark. 
\begin{lemma}\label{lem:grids}
Let $((M_\alpha,\CA_\alpha,\lambda_\alpha),\alpha) \in \widehat\Sigma$ with grid $T=\{t_i,i\geq 0\}$, and $\bar T=\{\bar t_i, i \geq 0\}$ a grid finer than $T$. There exists a mixed strategy $((M_\alpha,\CA_\alpha,\lambda_\alpha),\bar \alpha)$ with grid $\bar T$ such that: 
\[ \forall \xi_\alpha \in M_\alpha, \forall \omega \in \Omega, \forall v \in \CV, \; \alpha(\xi_\alpha,\omega,v) =\bar \alpha(\xi_\alpha,\omega,v).\]
The same is true for mixed strategies of Player $2$.
\end{lemma}
\begin{proof}
Let $\bar \alpha = \sum_{q \geq 0} \indic_{(\bar t_q,\bar t_{q+1}]} \bar \alpha^q$, where for all $q\geq 0$ we define 
\[\forall (\omega,v) \in \Omega_{t_q}\times \CV_{t_q}, \; \bar\alpha^q (\omega,v)= \alpha^n(\omega|_{[0,t_n]},v|_{[0,t_n]}).\]
where $n$ is the unique integer such that $(\bar t_q, \bar t_{q+1}] \subset (t_n,t_{n+1}]$.
The verification that the maps $\alpha$ and $\alpha'$ coincide is straightforward. The proof for strategies of player 2 is similar.
\end{proof}

The main advantage of the non-anticipative strategies with grids (or with delay) is that we may define the game using strategies against strategies rather than using strategies against controls. This is due to the following standard result.
\begin{lemma}\label{fixed_point}
For all pairs of pure strategies $(\alpha,\beta)\in \Sigma \times \CT$ and for all $\omega \in \Omega$, there exists a unique pair $(u^{\alpha,\beta},v^{\alpha,\beta})(\omega)\in \CU \times \CV$ such that
\begin{equation}\label{eq:fixed_point}
\forall \omega \in \Omega, \; v^{\alpha,\beta}(\omega)=\beta(u^{\alpha,\beta}(\omega)), \; u^{\alpha,\beta}(\omega)=\alpha(\omega,v^{\alpha,\beta}(\omega)).
\end{equation}
The process $(\omega,t)\in \Omega \times [0,\infty) \rightarrow (u^{\alpha,\beta},v^{\alpha,\beta})(\omega)_t$ is $\FF^X$-adapted,
left-continuous and piecewise constant on the grid $T=\{t_i, i \geq 0\}$ obtained by taking the union the two grids associated to $\alpha$ and $\beta$, and therefore $\FF^X$-predictable.
\p
For all pairs of mixed strategies $(\alpha,\beta)\in \widehat\Sigma \times \widehat \CT$, 
we use the notation
\[  (u^{\alpha,\beta},v^{\alpha,\beta})(\xi_\alpha,\xi_\beta,\omega)=(u^{\alpha(\xi_\alpha),\beta(\xi_\beta)},v^{\alpha(\xi_\alpha),\beta(\xi_\beta)})(\omega).\]
If $T=\{t_i, i \geq 0\}$ denotes a common grid to $\alpha$ and $\beta$, then for all $
i \geq 0$, the maps $(\xi_\alpha,\xi_\beta,\omega_{[0,t_i]}) \rightarrow (u^i,v^i)$ are measurable where $(u^i,v^i)$ denotes the value of $(u^{\alpha,\beta},v^{\alpha,\beta})$ on the interval $(t_i,t_{i+1}]$.
\end{lemma}
\begin{proof}
When $(\alpha,\beta)$ are pure strategies, we may assume thanks to Lemma \ref{lem:grids} that they have the same grid $T=\{t_i,i \geq \}$ obtained by taking the union the two grids associated to $\alpha$ and $\beta$. Define $(u^{\alpha,\beta},v^{\alpha,\beta})= \sum_{i \geq 0} \indic_{(t_i,t_{i+1}]} (u^i,v^i)$ where the maps $u^i:\Omega_{t_i} \rightarrow U$ and $v^i:\Omega_{t_{i}} \rightarrow V$ are defined by induction on $i \geq 0$ through the formulas 
\[ u^{i}(\omega|_{[0,t_i]})= \alpha^i \left(\omega_{[0,t_i]},\sum_{m=0}^{i-1} \indic_{(t_m,t_{m+1}]}v^m(\omega|_{[0,t_m]})\right),\]
\[ v^{i}(\omega|_{[0,t_{i}]})= \beta^i\left(\sum_{m=0}^{i-1} \indic_{(t_m,t_{m+1}]}u^m(\omega|_{[0,t_m]})\right).\]
That $(u^{\alpha,\beta},v^{\alpha,\beta})$ is the unique solution of \eqref{eq:fixed_point} follows by noticing that \eqref{eq:fixed_point} is equivalent to the above system of equations defining the maps $(u^i,v^i)$. The other properties follow directly from the definition.
\p
When $(\alpha,\beta)$ are mixed strategies, we have $(u^{\alpha,\beta},v^{\alpha,\beta})= \sum_{i \geq 0} \indic_{(t_i,t_{i+1}]} (u^i,v^i)$ where the maps $u^i:M_\alpha\times M_\beta\times \Omega_{t_i} \rightarrow U $ and $v^i:M_\alpha\times M_\beta\times\Omega_{t_{i}} \rightarrow V$ are defined by induction on $i \geq 0$ through the formulas 
\[ u^{i}(\xi_\alpha,\xi_\beta,\omega|_{[0,t_i]})= \alpha^i \left(\xi_\alpha,\omega_{[0,t_i]},\sum_{m=0}^{i-1} \indic_{(t_m,t_{m+1}]}v^m(\xi_\alpha,\xi_\beta,\omega_{[0,t_m]})\right),\]
\[ v^{i}(\xi_\alpha,\xi_\beta,\omega|_{[0,t_{i}]})= \beta^i\left(\xi_\beta,\sum_{m=0}^{i-1} \indic_{(t_m,t_{m+1}]}u^m(\xi_\alpha,\xi_\beta,\omega_{[0,t_m]})\right).\]
The required measurability property follows therefore by composition.
\end{proof}

\subsection{Construction of controlled Markov chains}

The next lemma shows how to construct a controlled Markov chain associated to any pair of controls of the players, and lists the important properties that will be used in section \ref{section_DPP} to prove dynamic programming inequalities.

\begin{lemma}\label{girsanov}  \
\begin{enumerate}
\item 
For all $p\in \Delta(K)$ and  all $\FF^X$-predictable controls $(u,v)$ with values in $U\times V$ that are left-continuous and piecewise-constant over a grid $T$, there exists a 
probability $\PP_p^{u,v}$ on the space $\Omega$ such that the canonical process $X$ is a controlled jump process with initial law $p$ and $\FF^X$-intensity $R(u,v)$. Moreover, $\PP_p^{u,v}$ has the following properties:
\begin{itemize}
\item[a)] $\PP_p^{u,v}= \sum_{k \in K}p_k\PP_{\delta_K}^{u,v}$,
\item[b)] The process $e^{-\int_0^t \tr R(u_s,v_s)ds}\delta_{X_t}$ is a $(\PP_p^{u,v},\FF^X)$ martingale where for a matrix $M$, $\tr M$ denotes its transpose, 
\item[c)] For all $\varepsilon\geq 0$, define $(\bar X_s)_{s \geq 0}=(X_{\varepsilon+s})_{s \geq 0}$ and for all $\omega_\varepsilon \in \Omega_\varepsilon$, let the $\FF^{X}$-predictable processes $(\bar u (\omega_\varepsilon)), \bar v (\omega_\varepsilon))$ be defined by 
\[ \forall \omega \in \Omega, \forall t \geq 0, (\bar u (\omega_\varepsilon),\bar v (\omega_\varepsilon))(\omega,t) = (u,v)(\omega_\varepsilon \oplus_{\varepsilon-} \omega ,t).\]
where $\omega_\varepsilon \oplus_{\varepsilon-} \omega (t) = \indic_{[0,\varepsilon)}(t) \omega_\varepsilon(t) + \indic_{[\varepsilon,\infty)}(t) \omega(t-\varepsilon)$.
Then, the conditional law of $\bar X$ given $\CF^X_\varepsilon$ under $\PP_p^{u,v}$ is $\PP^{\bar u(X|_{[0,\varepsilon]}),\bar v(X|_{[0,\varepsilon]})}_{\delta_{X_\varepsilon}}$. 
\end{itemize} 
Given any pair of strategies $(\alpha,\beta) \in \Sigma\times \CT$, we use the notation $\PP_p^{\alpha,\beta}=\PP_p^{u^{\alpha,\beta},v^{\alpha,\beta}}$.

\item
For all $(\alpha,\beta) \in \widehat{\Sigma}\times \widehat{\CT}$ the map 
\[ (\xi_\alpha,\xi_\beta) \rightarrow \PP^{\alpha(\xi_\alpha),\beta(\xi_\beta)}_p ,\]
is a transition probability from $(M_\alpha\times M_\beta, \CA_\alpha \otimes \CA_\beta)$ to $(\Omega,\CF)$ 
and we define the probability $\PP^{\alpha,\beta}_p$ on $M_\alpha\times M_\beta \times\Omega$ by 
\[ \forall A\in \CA_{\alpha}\otimes \CA_\beta\otimes \CF, \;  \PP^{\alpha,\beta}_p(A)=\int_{M_\alpha}\int_{M_\beta} \EE^{\alpha(\xi_\alpha),\beta(\xi_\beta)}_p[\indic_A] d\lambda_\alpha(\xi_\alpha) d\lambda_\beta(\xi_\beta) .\]
\end{enumerate}
\end{lemma}

\begin{remark}
Note that we may also define the law $\PP_p^{u,v}$ by concatenation since we consider only piecewise-constant controls. We prefer the construction given in the proof below as it can be easily generalized to a larger class of controls that are not piecewise-constant and is actually simpler to manipulate.
\end{remark}

\begin{proof}
We starts with the proof of $1)$. Consider the intensity matrix $R^0=(R^0_{i,j})_{i,j \in K}$ defined by $R^0_{i,j}=1$ whenever $i\neq j$ and $R^0_{i,i}=-(|K|-1)$ for all $i \in K$. It is well-known that there exists a probability $\PP_p$ on $(\Omega,\CF)$ under which the canonical process $(X_t)_{t \geq 0}$ is a Markov chain with initial law $p$ and transition matrix $R^0$. Moreover, the Markov property implies that $\PP_p=\sum_{k \in K}p_k \PP_{\delta_k}$.

For all $(i,j)\in K^2$ such that $i\neq j$, define the process $N^{i,j}$ by 
\[ \forall t  \geq 0, \; N^{i,j}_t= \sum_{0 < s \leq t} \indic_{X_{t-}=i}\indic_{X_t=j},\]
which counts the number of jumps of $X$ from $i$ to $j$. 
This process is a counting process with $(\PP_p,\FF^X)$-intensity $(\indic_{X_t=i})_{t \geq 0}$ (see e.g. chapter I in Bremaud \cite{Bremaud}).
Note that for all $t\geq 0$, $\sum_{(i,j)\in K^2 \,:\, i\neq j} N^{i,j}_t <\infty$ for all $\omega$ since we work on the space of c\`{a}dl\`{a}g trajectories taking values in a finite set.

Thanks to the assumptions on $(u,v)$, the process $R(u,v)=(R(u,v)_{i,j})_{(i,j)\in K^2}$ is $\FF^X$-predictable and bounded.
Define the density process $L^{u,v}(X)$ by 
\[\forall t \geq 0, \;  L^{u,v}_t =\prod_{(i,j)\in K^2 \,:\, i\neq j} L_t^{u,v,i,j},\]
where 
\[ L^{u,v,i,j}_t = \exp\left( \int_0^t \indic_{X_s=i}(1-R(u_s,v_s)_{i,j})ds \right) \prod_{s \in (0,t] \;:\; \Delta N^{i,j}_s >0} R(u_s,v_s)_{i,j}.\]
According to theorems T2,T4 chapter VI.2 in \cite{Bremaud}, the process $L^{u,v}$ is a $(\PP_p,\FF^X)$ martingale. 
We may therefore apply theorem 4.1 p.141 in Parthasarathy \cite{Part}, which implies that there exists a unique probability  $\PP_p^{u,v}$ on $(\Omega,\CF)$ satisfying
\[ \forall t\geq 0, \; \frac{d\PP^{u,v}_p}{d\PP_p}|_{\CF^X_t}=L^{u,v}_t.\]
Applying theorem T3 chapter VI.2 in \cite{Bremaud}, we deduce that the probability $\PP_p^{u,v}$ is such that $\PP_p^{u,v}$ is the law of a controlled jump process with initial law $p$ and $\FF^X$-intensity $R(u,v)$.

Property $a)$ follows therefore directly from the definition of $\PP_p^{u,v}$ together with the corresponding property for $\PP_p$.

Let us prove point $b)$. Consider the matrix valued $\FF^X$-predictable process $t\rightarrow e^{- \int_0^t \tr R(u_s,v_s)ds}$ and let $(y_{k,\ell}(t) )_{(k,\ell) \in K^2}$ denote its coordinates. For all $(k,\ell) \in K^2$, we have 
\[ \forall t\geq 0, \; y_{k,\ell}(t) = \indic_{k=\ell}- \int_0^t \sum_{i\in K} y_{k,i}(s) \tr R(u_s,v_s)_{i,\ell}  ds.\]
Using that $R(u_s,v_s)$ is a transition matrix, we have $\tr R(u_s,v_s)_{\ell,\ell}= - \sum_{i \in K \,:\,i \neq \ell }\tr R(u_s,v_s)_{i,\ell} $ and therefore
\[  \forall t\geq 0, \;y_{k,\ell}(t) = \indic_{k=\ell}- \int_0^t  \sum_{i \in K \,:\,i \neq \ell } \tr R(u_s,v_s)_{i,\ell} (y_{k,i}(s)-y_{k,\ell}(s)) ds\]
From this equality, we deduce that
\[ \forall t\geq 0, \; y_{k,X_t}(t)= \indic_{k=X_0} + \int_{0}^t \sum_{(i,\ell)\in K^2 \,:\, i\neq \ell}(y_{k,i}(s)-y_{k,\ell}(s)) (dN^{i,\ell}_s - \tr R(u_s,v_s)\indic_{X_s=\ell}ds). \] 
Applying theorem T6 chapter I.3 in \cite{Bremaud}, we deduce that the process 
$M_t= e^{-\int_0^t \tr R(u_s,v_s)ds}\delta_{X_t}=(y_{k,X_t})_{k \in K}$ is a $(\PP_p^{u,v},\FF^X)$ martingale.

Let us prove point $c)$. Recall that under the probability $\PP_p$, the conditional law of $\bar X$ given $\CF^X_\varepsilon$ is $\PP_{\delta_{X_\varepsilon}}$ thanks to the Markov property. 
From the definition of $L^{u,v}$, we have with obvious notations
\[ L^{u,v}_{\varepsilon+s}(X)= L^{u,v}_\varepsilon (X) L^{\bar u(X|_{[0,\varepsilon]}),\bar v(X|_{[0,\varepsilon]})}_s (\bar X) .\]
Using the formula for conditional expectations and densities, we have for all $T \geq 0$ and all $A \in \CF^X_T$
\begin{align*} \EE_p^{u,v}[\indic_A(\bar X) |\CF^X_\varepsilon ] &=\EE_p[\indic_A(\bar X) L^{\bar u(X|_{[0,\varepsilon]}),\bar v(X|_{[0,\varepsilon]})}_T (\bar X) |\CF^X_\varepsilon ]\\
&= \int_{\Omega}  \indic_A \times L^{\bar u(X|_{[0,\varepsilon]}),\bar v(X|_{[0,\varepsilon]})}_T d\PP_{\delta_{X_\varepsilon}} \\
&=\PP_{\delta_{X_\varepsilon}}^{\bar u(X|_{[0,\varepsilon]}),\bar v(X|_{[0,\varepsilon]})}(A).
\end{align*}
This equality can be extended to all $A\in \CF$ by a monotone class argument and this proves the result.

Let us prove $2)$. Consider a pair of mixed strategies $(\alpha,\beta)$. Thanks to Lemma \ref{lem:grids}, we may we assume that they have a the same grid $T=\{t_i,i \geq 0\}$.
Let $(u^n,v^n)(\xi_\alpha,\xi_\beta,\omega)$ denote the  value of $(u^{\alpha,\beta},v^{\alpha,\beta})$ on the interval $(t_n,t_{n+1}]$. We have for all $A\in \CF^X_t$ with $t\in (t_n,t_{n+1}]$
\begin{align*}
\PP^{\alpha(\xi_\alpha),\beta(\xi_\beta)}_p[A]&= \EE_p[L^{\alpha,\beta}_{t_{n+1}}\indic_A(X)], 
\end{align*}
where the variables $L^{\alpha,\beta}_{t_n} (\xi_\alpha,\xi_\beta,\omega)$ are defined by induction by $L^{\alpha,\beta}_0=1$ and 
\begin{multline*} L^{\alpha,\beta}_{t_{n+1}} 
= L^{\alpha,\beta}_{t_n} \exp \left(\sum_{i \in K} \int_{t_n}^{t_{n+1}} \left[|K|-1 - \sum_{j \neq i} R(u^{n},v^{n})_{i,j}  \right] \indic_{X_s=1} ds \right) \\ \times \prod_{t_n<s\leq t_{n+1} \,:\, X_{s-}\neq X_s} R(u^{n},v^{n})_{X_{s-},X_s}.
\end{multline*} 
The above expression, together with lemma \ref{fixed_point} and Fubini theorem implies that 
\[ (\xi_\alpha,\xi_\beta) \rightarrow \PP^{\alpha(\xi_\alpha),\beta(\xi_\beta)}_p[A] ,\]
is Borel measurable. By a monotone class argument, this property extends to all $A\in \CF$ and therefore the above map is a well-defined transition probability from $M_\alpha\times M_\beta$ to $(\Omega,\CF)$.
\p
We can therefore define the probability $\PP^{\alpha,\beta}_p$ on $M_\alpha\times M_\beta \times\Omega$ by 
\[ \forall A\in \CA_{\alpha}\otimes \CA_\beta\otimes \CF, \;\PP^{\alpha,\beta}_p(A)=\int_{M_\alpha}\int_{M_\beta} \PP^{\alpha(\xi_\alpha),\beta(\xi_\beta)}_p[A] d\lambda_\alpha(\xi) d\lambda_\beta(\zeta) .\]
\end{proof}

\subsection{Payoffs}

\begin{definition} 
For  all $p\in \Delta(K)$ and all $(\alpha,\beta) \in \widehat{\Sigma}\times \widehat{\CT}$, we define
\begin{align*}
 J(p,\alpha,\beta)&= \EE^{\alpha,\beta}_p[\int_0^\infty r e^{-rt}g(X_t,u^{\alpha,\beta}_t,v^{\alpha,\beta}_t)dt]\\
 &=\int_{M_\alpha}\int_{M_\beta} \EE^{\alpha(\xi_\alpha),\beta(\xi_\beta)}_p[\int_0^\infty r e^{-rt}g(X_t,u^{\alpha,\beta}_t,v^{\alpha,\beta}_t)dt]d\lambda_\alpha(\xi_\alpha) d\lambda_\beta(\xi_\beta).
\end{align*}
\end{definition}

\p
We define the lower and upper value functions of the game by:
\[ \forall p \in \Delta(K), \; W^- (p) = \sup_{\alpha \in \widehat\Sigma}\; \inf_{\beta \in \widehat \CT} \; J(p,\alpha,\beta)\]
\[\forall p \in \Delta(K), \;  W^+ (p) =  \inf_{\beta \in \widehat \CT} \; \sup_{\alpha \in \widehat\Sigma}\; J(p,\alpha,\beta)\]
We always have $W^-\leq W^+$ and the game is said to have a value $W$ if 
\[W=W^-=W^+.\]
\p
\textbf{Isaacs condition}
\p
We assume that the value $H(p,z)$ of the ``infinitesimal game'' with symmetric information and prior $p$ exists, i.e. for all $(p,z) \in \Delta(K) \times \RR^K$:
\begin{align}
 H(p,z) &= \sup_{u \in U} \; \inf_{v \in V} \; \langle \tr R(u,v) p, z\rangle + rg(p,u,v) \\
 & =\inf_{v \in V}\; \sup_{u \in U} \;  \langle \tr R(u,v) p, z\rangle + rg(p,u,v), \notag
\end{align}  
where $g(p,u,v)=\sum_{k \in K} p_k g(k,u,v)$.
\p
The following lemma collects standard properties of $W^-$, $W^+$ and $H$.
\begin{lemma}\label{lem:properties}
We have for all $p\in \Delta(K)$:
\[ W^- (p) = \sup_{\alpha \in \widehat\Sigma}\; \inf_{\beta \in \CT} \; J(p,\alpha,\beta)\]
\[ W^+ (p) =  \inf_{\beta \in \widehat \CT} \; \sup_{\alpha \in \Sigma}\; J(p,\alpha,\beta).\]
$W^+$ and $W^-$ are concave and $\sqrt{|K|}$-Lipschitz functions. 
\p
There exists a constant $C$ such that all $z,z' \in \RR^K$, and $p,p' \in \Delta(K)$
\[ |H(p,z) - H(p',z')| \leq C(|z-z'|+|z||p-p'|).\]
\end{lemma}
\begin{proof}
The first two equalities follow from standard arguments.
\p
For all $(\alpha,\beta) \in \widehat{\Sigma}\times \widehat{\CT}$ and $p,p'\in \Delta(K)$, we have 
\[ |J(p,\alpha,\beta) - J(p',\alpha,\beta)| \leq \sum_{k \in K} |p_k-p'_k| |J(\delta_k,\alpha,\beta)| \leq \sqrt{|K|}|p-p'|,\]
where we used that $\PP_p^{u,v}=\sum_{k \in K}p_k\PP_{\delta_k}^{u,v}$ and $\|g\|_\infty \leq 1$. The fact that $W^-$ and $W^+$ are $\sqrt{|K|}$ Lipschitz follows then from standard arguments. 
\p
Let us prove that $W^+$ is concave. For all $p\in \Delta(K)$, we have 
\begin{align*}
W^+ (p) &=  \inf_{\beta \in \widehat \CT} \; \sup_{\alpha \in \Sigma}\; J(p,\alpha,\beta)\\
&= \inf_{\beta \in \widehat \CT} \; \sup_{\alpha \in \Sigma}\; \sum_{k \in K}p_k J(\delta_k,\alpha,\beta) \\
& \leq \inf_{\beta \in \widehat \CT} \; \sum_{k \in K}p_k \sup_{\alpha \in \Sigma}\; J(\delta_k,\alpha,\beta).
\end{align*}
We claim that the last inequality is actually an equality. Indeed, let $\varepsilon>0$ and $\beta \in \widehat\CT$.  For all $k\in K$, let $\alpha^k\in \Sigma$ such that
\[J(\delta_k,\alpha^k,\beta) \geq \sup_{\alpha \in \Sigma}\; J(\delta_k,\alpha,\beta) - \varepsilon \]
Define a strategy $\bar\alpha \in \Sigma$ by 
\[\forall \omega \in \Omega, \forall v\in \CV, \; \bar\alpha(\omega,v) =\sum_{k \in K} \indic_{\omega(0)=k} \alpha^k(\omega,v).\]
Note that $\bar\alpha$ is a well-defined strategy since $\omega(0)$ is a measurable map of $\omega|_{[0,t]}$ for all $t$. With this definition, we have for all $k\in K$
\begin{align*} J(\delta_k,\bar\alpha,\beta)&=\EE_{\delta_k}^{\bar\alpha,\beta}[\int_0^\infty re^{-rt}g(X_t,u^{\bar\alpha,\beta}_t,v^{\bar\alpha,\beta}_t)dt] \\
&=\EE_{\delta_k}^{\alpha^k,\beta}[\int_0^\infty re^{-rt}g(X_t,u^{\alpha^k,\beta}_t,v^{\alpha^k,\beta}_t)dt] \\
&= J(\delta_k,\alpha^k,\beta),
\end{align*}
since the processes $(u^{\bar\alpha,\beta}_t,v^{\bar\alpha,\beta}_t)(\xi_\beta)$ and $(u^{\alpha^k,\beta}_t,v^{\alpha^k,\beta}_t)(\xi_\beta)$ are equal $\PP_{\delta_k}$ almost surely for all $\xi_\beta \in M_\beta$.  
It follows that
\begin{align*}
J(p,\bar\alpha,\beta)= \sum_{k \in K}p_k J(\delta_k,\bar\alpha,\beta) = \sum_{k \in K} J(\delta_k,\alpha^k,\beta)\geq \sum_{k \in K}p_k \sup_{\alpha \in \Sigma}\; J(\delta_k,\alpha,\beta) - \varepsilon,
\end{align*}
and the claim follows by sending $\varepsilon$ to zero.
To conclude, note that 
\[ W^+(p)= \inf_{\beta \in \widehat \CT} \; \sum_{k \in K}p_k \sup_{\alpha \in \Sigma}\; J(\delta_k,\alpha,\beta),\]
and thus $W^+$ is concave as an infimum of affine maps.
\p
Let us prove that $W^-$ is concave. The proof relies on the classical splitting method. 
Let $p_1,p_2 \in \Delta(K)$ and $s\in[0,1]$, and define $p=sp^1+(1-s)p^2$. Let $\varepsilon>0$ and for $i=1,2$, let $\alpha_i \in \widehat\Sigma$ such that
\[ \inf_{\beta \in \CT} \; J(p^i,\alpha_i,\beta) \geq W^-(p^i) - \varepsilon.\]
We define now mixed strategies $\bar\alpha$ and $(\bar\alpha_k)_{k \in K}$ having the same probability space $(M_{\bar\alpha},\CA_{\bar\alpha},\lambda_{\bar\alpha})$ defined by  
\[M_{\bar\alpha}=[0,1]\times M_{\alpha_1} \times M_{\alpha_2}, \; \CA_{\bar\alpha}=\CB([0,1])\otimes \CA_{\alpha_1}\otimes \CA_{\alpha_2}, \; \lambda_{\bar\alpha}=\Leb\otimes \lambda_{\alpha_2}\otimes \lambda_{\alpha_2},\] 
where $\Leb$ denotes the Lebesgue measure on $[0,1]$.
A typical element of $M_\alpha$ will be denoted $(\zeta,\xi_{\alpha_1},\xi_{\alpha_2})$.
For all $k\in K$, the strategy $\bar\alpha_k$ is defined by 
\[\bar\alpha_k(\zeta,\xi_{\alpha_1},\xi_{\alpha_2},\omega)= \indic_{\zeta\leq m_k} \alpha_1(\xi_{\alpha_1},\omega)+ \indic_{\zeta >m_k}\alpha_2(\xi_{\alpha_2},\omega),\]
and $\bar\alpha$ is defined by
\[ \bar\alpha(\zeta,\xi_{\alpha_1},\xi_{\alpha_2},\omega)= \sum_{k \in K}\indic_{\omega(0)=k} \bar\alpha_k(\zeta,\xi_{\alpha_1},\xi_{\alpha_2},\omega),\]
where the numbers $(m_k)_{k \in K}$ are defined by 
\[ m_k= \frac{s p^1_k}{p_k} \indic_{p_k >0}. \]
As above, with this definition, we have for all $\beta \in \CT$ and all $k\in K$,
\[ J(\delta_k,\bar\alpha,\beta)= J(\delta_k,\bar\alpha_k,\beta).\]
It follows that (integrating with respect to $\zeta$)
\begin{align*}
J(p,\bar\alpha,\beta)&= \sum_{k \in K}p_k J(\delta_k,\bar\alpha_k,\beta) \\
&= \sum_{k \in K}p_k (m_k J(\delta_k,\alpha_1,\beta)+ (1-m_k) J(\delta_k,\alpha_2,\beta))\\
&= s J(p^1,\alpha_1,\beta)+ (1-s) J(p^2,\alpha_2,\beta).
\end{align*}
We obtain
\begin{align*}
W^-(p) &\geq   \inf_{\beta \in \CT} \; J(p,\bar\alpha,\beta) \\
 &=  \inf_{\beta \in \CT} \;  \left[ s J(p^1,\alpha_1,\beta)+ (1-s) J(p^2,\alpha_2,\beta) \right] \\
& \geq  s \inf_{\beta \in \CT} \;   J(p^1,\alpha_1,\beta) + (1-s) \inf_{\beta \in \CT} \; J(p^2,\alpha_2,\beta)\\
&\geq s W^-(p^1) +(1-s) W^-(p^2) - \varepsilon,
\end{align*}
and the proof follows by sending $\varepsilon$ to zero.
\p
The last statement follows from the fact that for all $(u,v) \in U\times V$, all $z,z' \in \RR^K$, and all $p,p' \in \Delta(K)$
\begin{align*}
\left|\langle z, \tr R(u,v) p \rangle- \langle z', \tr R(u,v) p' \rangle \right| &\leq \left|\langle z, \tr R(u,v) p \rangle - \langle z, \tr R(u,v) p' \rangle \right| \\&
\quad +\left|\langle z, \tr R(u,v) p' \rangle- \langle z', \tr R(u,v) p' \rangle \right| \\
&\leq |z| \|R\|_\infty |p-p'|+ \sqrt{|K|} \|R\|_\infty |z-z'|.
\end{align*}
\end{proof}

\subsection{Main results}

For all $p\in \Delta(K)$, let 
\[TS_{\Delta(K)}(p)= \{ y \in \RR^K \tq \exists \varepsilon>0, p+\varepsilon y \in \Delta(K), p-\varepsilon y \in \Delta(K) \}\]
denote the tangent space to $\Delta(K)$ at $p$. 
Given any $K\times K$ symmetric matrix $A$, define 
\[\lambda_{\max}(p,A)= \sup \left\{ \frac{\langle A y, y \rangle}{|y|^2}, y\in TS_{\Delta(K)}(p)\setminus\{0\} \right\},\]
which is the maximal eigenvalue of the restriction of $A$ to $TS_{\Delta(K)}(p)$ with the convention $\lambda_{\max}(p,A)=-\infty$ if $TS_{\Delta(K)}(p)=\{0\}$.

We consider the following Hamilton Jacobi equation, introduced in \cite{cardaetal}, with unknown $f:\Delta(K)\rightarrow \RR$
\begin{equation}\label{eq:HJB}
\forall p \in \Delta(K), \;  \min \{ r f(p) - H(p,\nabla f(p)) \,;\, - \lambda_{\max}(p,D^2f(p))\}=0,
\end{equation}
where $\nabla$ denotes the gradient and $D^2$ the Hessian matrix.
Let us give a precise definition of a viscosity solution of \eqref{eq:HJB}.
\begin{definition}\
\begin{enumerate}
\item A function $f:\Delta(K)\mapsto\RR$ is called a supersolution of \eqref{eq:HJB} if it is lower semi-continuous and satisfies: for any smooth test function $\varphi:\Delta(K)\mapsto\RR$ and $p\in\Delta(K)$ such that $\varphi-f$ has a global maximum at $p$, we have
\[ \min\left\{ r \varphi(p)-H(p,\nabla \varphi(p));-\lambda_{\max}(p,D^2\varphi(p))\right\}\geq 0 .\]
\item A function $f:\Delta(K)\mapsto\RR$ is called a subsolution of \eqref{eq:HJB} if it is upper semi-continuous and satisfies: for any smooth test function $\varphi:\Delta(K)\mapsto\RR$ and $p\in\Delta(K)$ such that  $\varphi-f$ has a global minimum at $p$, we have
\[ \min\left\{ r \varphi(p)-H(p,\nabla \varphi(p));-\lambda_{\max}(p,D^2\varphi(p))\right\}\leq 0 .\]
\end{enumerate} 
A function $f:\Delta(K)\mapsto\RR$ is called a solution of \eqref{eq:HJB} if it is both a supersolution and a subsolution.
\end{definition}

\begin{theorem}\label{th:viscosity}
Under Isaacs assumption, the value $W$ exists and is the unique Lipschitz viscosity solution on $\Delta(K)$ of \eqref{eq:HJB}.
\end{theorem}

This result has to be compared with the main result in \cite{cardaetal}, in which the authors proved that the limit value obtained through an approximating sequence of discrete-time games is the unique viscosity of the above equation. We therefore provide an equivalent result for the continuous-time model.

Our second contribution is to obtain a new variational characterization of the value, which is roughly speaking a pointwise version of the above Hamitlon-Jacobi equation based on directional derivatives.
One of the main interest of this new formulation is that the comparison principle is very simple to prove and avoids all the technical machinery of the viscosity solution that was used in \cite{cardaetal}  to obtain the same result (inf/sup convolutions, doubling of variables, Jensen's Lemma, etc..).

Let $f: \Delta(K) \mapsto \RR$ be a concave Lipschitz function, $p\in \Delta(K)$ and $z \in T_{\Delta(K)}(p)$, where $T_{\Delta(K)}(p)$ denotes the tangent cone of $\Delta(K)$ at $p$. Then the directional derivative of $f$ at $p$ in the direction $z$ defined by 
\[ \vec  D f (p ; z) = \lim_{\varepsilon \downarrow 0} \frac{1}{\varepsilon}( f(p+\varepsilon z) - f(p))\]
exists and is finite. Let $Exp(f)$ denotes  the set of exposed points of $f$, i.e. the set of $p\in \Delta(K)$ such that there exists $x\in \RR^K$ such that
\[ \argmin_{p' \in \Delta(K)} \; \langle x , p' \rangle - f(p')  = \{p\}.\]  
\begin{theorem}\label{th:variational}
$W$ is the unique concave Lipschitz function such that
\begin{equation}\label{sup_var} 
\forall p \in \Delta(K), \;r W(p) -  \inf_{v \in V} \;\sup_{\mu \in \Delta(U)} \; \vec D W \left(p; \int_U \tr R(u,v) d\mu(u) p \right) + \int_U rg(p,u,v)d\mu(u) \geq 0
\end{equation}
\begin{equation}\label{sub_var}  \forall p \in Exp(W),\; r W(p) -  \inf_{v \in V} \; \sup_{\mu \in \Delta(U)} \; \vec D W \left(p; \int_U \tr R(u,v) d\mu(u) p \right) + \int_U rg(p,u,v)d\mu(u)\leq 0.
\end{equation}
\end{theorem}

\subsection{Generalizations and open questions}

As in \cite{cardaetal}, the present results can be extended to a zero-sum differential game where each player controls and observes privately his own continuous-time Markov chain.

In Gensbittel \cite{gensbittel2013} and in Gensbittel and Rainer \cite{obsbrown}, different models were analyzed through an approximating sequence of discrete-time game. The main difficulty to adapt the present method to these models lies in the difficulty to extend the duality techniques applied to first-order equations to second-order equations. Therefore, the direct analysis of these models in continuous-time remains a challenging problem.

\section{On the new formulation of the Hamitlon-Jacobi equation}\label{section_HJB}

At first, the next lemma explains why the set $\Delta(U)$ appears in the inequalities \eqref{sup_var} and \eqref{sub_var}.

\begin{lemma}\label{lem:minmaxH}
Let $f: \RR^K \rightarrow \RR$ be a concave Lipschitz function. Then, for all $p\in \Delta(K)$, we have
\begin{align*}
\min_{x \in \partial^+f(p)} H(p,x)&= \inf_{v \in V} \;  \sup_{\mu \in \Delta(U)}\; \vec D f \left(p ; \int_U \tr R(u,v) d\mu(u) p \right) + \int_U rg(p,u,v)d\mu(u).
\end{align*}
\end{lemma} 
\begin{proof}
The map 
\[ (x,\mu) \in \partial^+f(p)\times \Delta(U) \rightarrow \left\langle x, \int_U \tr R(u,v) d\mu(u) p \right\rangle + \int_Ur g(p,u,v)d\mu(u),\]
is bilinear and continuous with respect to $x$, $\partial^+f(p)$ is a compact convex set and $\Delta(U)$ is a convex set. Therefore, an extension of Sion's minmax theorem (see e.g. \cite{Sorin}) implies that 
\begin{align*}
\min_{x \in \partial^+f(p)} H(p,x)&=  \inf_{v \in V} \; \min_{x \in \partial^+f(p)} \; \sup_{u\in U}\; \langle x, \tr R(u,v) p \rangle + rg(p,u,v) \\ 
&=\inf_{v \in V} \;  \min_{x \in \partial^+f(p)} \; \sup_{\mu \in \Delta(U)}\; \left\langle x, \int_U \tr R(u,v) d\mu(u) p \right\rangle + \int_U rg(p,u,v)d\mu(u) \\ 
&= \inf_{v \in V} \;  \sup_{\mu \in \Delta(U)}\; \min_{x \in \partial^+f(p)} \left\langle x, \int_U \tr R(u,v) d\mu(u) p \right\rangle + \int_U rg(p,u,v)d\mu(u) \\
&= \inf_{v \in V} \;  \sup_{\mu \in \Delta(U)}\; \vec Df\left(p ; \int_U \tr R(u,v) d\mu(u) p\right) + \int_U rg(p,u,v)d\mu(u),
\end{align*}
where we also used that for all $(p,z)\in \RR^K\times \RR^K$, we have 
\[ \vec D f(p;z)= \min_{x \in \partial^+f(p)} \langle x, z \rangle.\]
\end{proof}

We prove below that any Lipshitz viscosity solution of \eqref{eq:HJB} is concave and satisfies \eqref{sup_var} and \eqref{sub_var}, and reciprocally that any Lipschitz concave function satisfying  \eqref{sup_var} and \eqref{sub_var} is a viscosity solution of \eqref{eq:HJB}. 

Then, in Proposition \ref{prop:comparison}, we will prove that there exists a unique concave Lipschitz function satisfying \eqref{sup_var} and \eqref{sub_var}. This provides therefore another proof that \eqref{eq:HJB} admits a unique Lipschitz viscosity solution which is shorter and simpler than the proof of the comparison principle given in \cite{cardaetal}. 

\subsection{Equivalence of the two variational characterizations}

We divide the proof of the equivalence in two propositions.

\begin{proposition}\label{prop:equiv_super}\ 
\begin{enumerate}
\item If   $f : \Delta(K) \mapsto \RR$ is a Lipschitz viscosity supersolution of \eqref{eq:HJB}, then $f$ is concave and satisfies \eqref{sup_var}. 
\item If $f : \Delta(K) \mapsto \RR$ is a concave Lipschitz function which satisfies \eqref{sup_var}, then $f$ is a viscosity supersolution of \eqref{eq:HJB}.
\end{enumerate}
\end{proposition}
\begin{proof}
Let us prove $1)$. In order to work on a convex set with non-empty interior, we denote by $\tilde f$ denote the restriction of $f$ to the affine space $A$ spanned by $\Delta(K)$.
The fact that $\tilde f$ is concave  on the relative interior of $\Delta(K)$ follows from   Lemma 1 in \cite{ALL} and the property extends to $\Delta(K)$ by continuity.
\p
$\tilde f$ is therefore concave  and Lipschitz on $\Delta(K) \subset A$ and its superdifferential is given by
\[\partial^+\tilde f(p)=\{ x \in E \tq \forall p' \in \Delta(K), \; \tilde{f} (p) + \langle x,p'-p \rangle \geq \tilde f (p')\},\]
where $E=A-A$ is the tangent space to $A$ and it is easily seen that $\partial^+f(p)= \partial^+ \tilde f(p) + \RR \cdot \gamma$ where $\gamma=(1,1,...,1)$ is a vector orthogonal to $A$ so that $\RR^K= A \oplus \RR \cdot \gamma$. 
Moreover, for all $p \in \Delta(K)$ and and all $z\in T_{\Delta(K)}(p)$, we have (see e.g. the appendix of \cite{GensbittelGrun} for the second equality)
\[ \vec D f(p;z)= \vec D \tilde{f}(p;z) = \min_{x \in \partial^+ \tilde f(p)} \langle x , z\rangle.\]
$\tilde{f}$ is differentiable at Lebesgue almost every $p$ in the relative interior of $\Delta(K)$ and its gradient $\nabla \tilde f(p) \in E$ is bounded by the Lipschitz constant of $f$. For any such $p$, it is well-known that the viscosity supersolution property implies that
\[ r \tilde f(p) \geq H(p,\nabla \tilde f(p)).\]
For any $p \in \Delta(K)$, there exists a sequence $p_n$ in the relative interior of $\Delta(K)$ with limit $p$ such that $\tilde f$ is differentiable at $p_n$ for all $n$. 
The sequence $\nabla \tilde f(p_n)$ being bounded, up to extract a subsequence, we may assume that $\nabla \tilde f(p_n) \rightarrow y \in \partial^+ \tilde f(p)$. 
We obtain
\begin{align*}
r \tilde f(p) = \lim_n r \tilde f(p_n) &\geq \lim_n  H(p_n,\nabla \tilde f (p_n)) =H(p,y).
\end{align*}
Since $y \in \partial^+ \tilde f(p)$, for all $z \in T_{\Delta(K)}(p)$, we have
\[ \langle y , z \rangle \geq \vec D \tilde f(p ; z) .\]
We deduce that
\begin{align*}
r\tilde f(p) \geq H(p,y) &= \inf_{v \in V} \; \sup_{\mu \in \Delta(U)} \;  \left\langle y, \int_U\tr R(u,v)d\mu(u) p \right\rangle + \int_U rg(p,u,v)d\mu(u) \\
&\geq  \inf_{v \in V} \; \sup_{\mu \in \Delta(U)} \; \vec D \tilde f \left(p ; \int_U\tr R(u,v)d\mu(u) p \right) + \int_U rg(p,u,v)d\mu(u),
\end{align*}
which concludes the proof.
\p
Let us prove $2)$. Assume that $\phi$ is a smooth test function such that $\phi \leq f$ on $\Delta(K)$ with equality at $p$. For any $z \in T_{\Delta(K)}(p)$, we have therefore
\[ \langle \nabla \phi(p) , z \rangle =\vec D \phi(p; z) \leq  \vec D f(p; z).\]
We deduce that
\begin{align*}
r f(p)  &\geq   \inf_{v \in V} \;\sup_{\mu \in \Delta(U)} \; \vec D f \left(p; \int_U \tr R(u,v) d\mu(u) p\right) + \int_U rg(p,u,v)d\mu(u) \\
& \geq \inf_{v \in V} \;\sup_{\mu \in \Delta(U)} \left\langle \nabla \phi(p),  \int_U \tr R(u,v) d\mu(u) p \right\rangle + \int_U rg(p,u,v)d\mu(u) \\
& \geq \inf_{v \in V} \;\sup_{u\in U} \; \langle \nabla \phi(p),  \tr R(u,v) p \rangle + rg(p,u,v) \\
&= H(p,\nabla\phi(p)),
\end{align*}
which concludes the proof.
\end{proof}

\begin{proposition}\label{prop:equiv_sub}
\ 
\begin{enumerate}
\item If  $f : \Delta(K) \mapsto \RR$ is a concave Lipschitz viscosity subsolution of \eqref{eq:HJB}, then $f$ is satisfies \eqref{sub_var}. 
\item If $f : \Delta(K) \mapsto \RR$ is a concave Lipschitz function which satisfies \eqref{sub_var}, then $f$ is a viscosity subsolution of \eqref{eq:HJB}.
\end{enumerate}
\end{proposition}
\begin{proof}
Let us prove $1)$. We first assume that $p \in \Delta(K)$ is such that there exists some smooth strongly concave (on a neighborhood of $\Delta(K)$)  map $\phi$ such that $\phi \geq f$ on $\Delta(K)$ and $\phi(p)=f(p)$. Since $\phi$ is strongly concave, there exists $\varepsilon>0$ such that $D^2\phi(p) \leq -\varepsilon I$ and thus $\lambda_{\max}(p,D^2\phi(p))<0$.
The viscosity subsolution property implies therefore that 
\[ r\phi(p) \leq H(p,\nabla \phi(p)).\]
Define for all $p' \in \Delta(K)$, $\psi(p')=\phi(p') - \phi(p) -  \langle \nabla \phi(p), p'-p \rangle$ and note that $\psi$ is strongly concave and that $\psi(p)=0$.  
We have
\[ \forall p' \in \Delta(K), f(p') \leq f(p) + \langle \nabla \phi(p), p'-p \rangle + \psi(p') .\]
Let $\hat f : \RR^K \rightarrow \RR$ denote the Moreau-Yosida regularization of $f$ defined by 
\begin{equation*}
 \forall y\in \RR^K, \; \hat f(y)= \sup_{p' \in \Delta(K)} f(p') - M|y-p'|,
\end{equation*}
for some constant $M$ larger than the Lipschitz constant of $f$. It is well-known that $\hat{f}$ is concave and $M$-Lipschitz on $\RR^K$ and coincides with $f$ on $\Delta(K)$ so that
\begin{equation} \label{eq:direct_coincides}
\forall p' \in \Delta(K), \forall z \in T_{\Delta(K)}(p'), \vec D f(p;z)= \vec D \hat f(p;z)= \min_{x \in \partial^+ \hat f(p')} \langle x,z \rangle,
\end{equation}  
where the set $\partial^+ \hat f(p')$ is a compact convex subset of $\partial^+  f(p')$ .
\p
Let $x\in \partial^+ \hat f(p)$ so that
\[ \forall p' \in \Delta(K), f(p') \leq f(p) + \langle x, p'-p \rangle.\]
For all $\lambda \in (0,1]$, we have
\[ \forall p' \in \Delta(K), f(p') \leq f(p) + \langle \lambda \nabla \phi(p)+(1-\lambda) x, p'-p \rangle + \lambda \psi(p'),\]
with equality at $p'=p$. Using the right-hand side of the above inequality which is strongly concave as a test function and applying the viscosity subsolution property, we deduce that 
\[ rf(p) \leq H( p,\lambda \nabla \phi(p)+(1-\lambda) x),\]
and letting $\lambda$ go to zero, we obtain $rf(p) \leq H(p,x)$. We conclude that
\begin{equation}\label{eq:ineq_strongly_exposed}
rf(p) \leq \min_{x \in \partial^+ \hat f(p)} H(p,x). 
\end{equation}
Note that the above inequality holds for any value of $M$ larger thant the Lipschitz constant of $f$.
\p
Let us now consider an arbitrary point $p\in Exp(f)$. 
Let $x \in\partial^+  f(p)$ such that
\[ \argmin_{p' \in \Delta(K)} \;\langle x, p' \rangle - f(p') = \{p\}.\]
Define $\hat f$ as above with $M \geq |x|+1$.
Let $y \in \partial^+ \hat f(p)$ and note that for all $\lambda \in (0,1]$, we have  
\[ \argmin_{p' \in \Delta(K)}\; \langle  y_\lambda, p' \rangle - f(p') = \{p\}.\]
with $y_\lambda= \lambda x +(1-\lambda) y$. For all $n\geq 1$, let 
\[p_n \in \argmin_{p' \in  \Delta(K)} \;\langle y_\lambda, p' \rangle - f(p')- \frac{1}{n}\ell(p'),\]
where $\ell(p')=\sqrt{1+|p'|^2}$. Note that the map $p' \rightarrow
 \langle y_\lambda, p' \rangle - \frac{1}{n}\ell(p')$ is strongly concave. 
By construction  $y_\lambda-\frac{1}{n}\nabla\ell(p_n) \in \partial^+f(p_n)$ and $p_n \rightarrow p$. Moreover, with our choice of $M$, we have $y_\lambda-\frac{1}{n}\nabla\ell(p_n) \in \partial^+ \hat f(p_n)$ for $n\geq \tfrac{1}{\lambda}$. Indeed, for such $n$ we have 
\[  |y_\lambda-\tfrac{1}{n}\nabla\ell(p_n)| \leq \lambda |x| +(1-\lambda) M +\tfrac{1}{n} \leq M - \lambda +\tfrac{1}{n} \leq M\]
and for all $z\in \RR^K$, there exists $p_z \in \Delta(K)$ such that
\begin{align*} 
\hat f(z) = f(p_z) - M|p_z-z| &\leq f(p_n)+ \langle y_\lambda-\tfrac{1}{n}\nabla\ell(p_n), p_z-p_n \rangle  -M|p_z-z| \\
&=\hat f(p_n)+ \langle y_\lambda-\tfrac{1}{n}\nabla\ell(p_n), p_z-p_n \rangle  -M|p_z-z| \\
&\leq \hat f(p_n)+ \langle y_\lambda-\tfrac{1}{n}\nabla\ell(p_n), z-p_n \rangle 
\end{align*}
which proves that $y_\lambda-\tfrac{1}{n}\nabla\ell(p_n) \in \partial^+ \hat f(p_n)$.
Using now \eqref{eq:ineq_strongly_exposed}, we have for all $n \geq \tfrac{1}{\lambda}$
\[ rf(p_n) \leq \min_{z \in \partial^+ \hat f(p_n)} H(p_n,z) \leq H(p_n,y_\lambda-\tfrac{1}{n}\nabla\ell(p_n)),\]
and therefore taking the limit as $n\rightarrow \infty$, we obtain $rf(p) \leq H(p,y_\lambda)$. By sending $\lambda$ to zero, we obtain $rf(p) \leq H(p,y)$ and the conclusion follows by taking the infimum over all $y \in \partial^+ \hat f(p)$ and then applying Lemma \ref{lem:minmaxH} and \eqref{eq:direct_coincides}.
\p
Let us prove $2)$. Let $\phi$ be a smooth test function such that $\phi \geq f$ with equality at $p \in \Delta(K)$ and $\lambda_{\max}(p,D^2\phi(p))<0$. 
Recall the definition of $\hat f$ in the proof of $1)$. By choosing $M \geq C=\sup_{y \in B} |\nabla  \phi(y) |$ where $B$ is a bounded neighborhood of $\Delta(K)$, we have $\phi \geq \hat f$ in $B$ and therefore $\nabla\phi(p) \in \partial^+ \hat f (p)$.
Indeed, if there exists $y \in B$ such that $\phi(y)< \hat f(y)$, then there exists $p_y \in \Delta(K)$ such that 
\[ \hat f(y)= f(p_y) - M|y-p_y| > \phi(y) \geq \phi(p_y) - C|y-p_y|  ,\]
which implies $f(p_y)>\phi(p_y)$ and thus contradicts the assumption. 
\p
Recall that $TS_{\Delta(K)}(p)$ denotes the tangent space of $\Delta(K)$ at $p$. Let $x \in \RR^K$ be a vector in the relative interior of the normal cone to $\Delta(K)$ at $p$ so that 
\[ \forall p' \in \Delta(K) \setminus (p+TS_{\Delta(K)}(p)), \; \; \langle x, p'-p \rangle <0.\]  
\[ \forall p' \in  (p+TS_{\Delta(K)}(p))\cap \Delta(K), \; \; \langle x, p'-p \rangle =0.\] 
We deduce that $\nabla\phi(p)-x \in \partial^+f(p)$.
On the other hand, since $\lambda_{\max}(p,D^2\phi(p))<0$, the map $\phi$ is  strongly concave on a neighborhood $\CO$ of $p$ in the affine space $p+TS_{\Delta(K)}(p)$ so that
\[ \forall p' \in \CO \cap \Delta(K), \;\; p'\neq p \Longrightarrow f(p') \leq \phi(p') < f(p) +\langle  \nabla \phi(p), p'-p \rangle.\]
Since $f$ is concave, we deduce that
\[ \forall p' \in (p+TS_{\Delta(K)}(p))\cap \Delta(K), \;\; p'\neq p \Longrightarrow f(p')  < f(p) +\langle  \nabla \phi(p), p'-p \rangle.\]
Combining the above inequalities, we obtain 
\[ \argmin_{p' \in \Delta(K)}\; \langle \nabla \phi(p) - x, p'-p \rangle - f(p') =\{p\},\]
which implies $p \in Exp(f)$. We deduce that 
\[ rf(p) \leq \min_{x \in \partial^+ \hat f(p)} H(p,x)  \leq H(p,D \phi(p)).\] 
\end{proof}

\subsection{Comparison principle}

Let us now prove a comparison principle for the new formulation of the equation.
The proof is quite simple and inspired by the proof of Mertens and Zamir \cite{MZ}. 

\begin{proposition}\label{prop:comparison}
Let $W_1$ and $W_2$ be concave Lipschitz functions from $\Delta(K)$ to $\RR$ such that $W_1$ satisfies \eqref{sup_var} and $W_2$ satisfies \eqref{sub_var}. Then $W_1 \geq W_2$.  
\end{proposition}
\begin{proof}
Assume by contradiction that $M=\max_{p \in \Delta(K)}W_2(p)-W_1(p)>0$. 
For all $\varepsilon>0$, define the perturbed problem 
\[ M_\varepsilon := \max_{p \in \Delta(K)} W_2(p)-(W_1(p)-\varepsilon \ell(p) ),\]
where $\ell(p)=\sqrt{1+|p|^2}$. 
Note that $\ell$ is a smooth Lipschitz function on $\RR^K$ and is strongly concave.
We have $ M \leq M_\varepsilon \leq M + \varepsilon C$ with $C=\sup_{p \in \Delta(K)}|\ell(p)|$. 
Let 
\[p_\varepsilon \in \argmax_{p \in \Delta(K)} \; W_2(p)-(W_1(p)-\varepsilon \ell(p) ).\] 
We claim that $p_\varepsilon$ is an exposed point of $W_2$. 
Note that by definition of $M_\varepsilon$, we have:
\[ \forall p\in \Delta(K), \; W_1(p)-\varepsilon \ell(p)+M_\varepsilon \geq W_2(p).\]
Let $y_\varepsilon \in \partial^+W_1(p_\varepsilon)$, then 
\[  \forall p\in \Delta(K), \; W_1(p_\varepsilon)+\langle y_\varepsilon, p-p_\varepsilon \rangle \geq W_1(p) .\]
We deduce that
\begin{equation}\label{ineq_exposed} \forall p\in \Delta(K), \; \phi_\varepsilon(p):=
W_1(p_\varepsilon)+\langle y_\varepsilon, p-p_\varepsilon \rangle-\varepsilon \ell(p)+M_\varepsilon \geq W_2(p).
\end{equation}
Note that $\phi_\varepsilon(p_\varepsilon)=W_2(p_\varepsilon)$ and that $\phi_\varepsilon$  is a smooth strongly concave function. Therefore, $p_\varepsilon$ is an exposed point of $W_2$.
Applying \eqref{sup_var} and \eqref{sub_var} at $p_\varepsilon$, we obtain
\[r W_1(p_\varepsilon) \geq  \inf_{v \in V} \;\sup_{\mu \in \Delta(U)} \; \vec D W_1 \left(p_\varepsilon; \int_U \tr R(u,v) d\mu(u) p_\varepsilon \right) + \int_U rg(p_\varepsilon,u,v)d\mu(u) ,\]
\[  r W_2(p_\varepsilon) \leq  \inf_{v \in V} \; \sup_{\mu \in \Delta(U)} \; \vec D W_2 \left(p_\varepsilon; \int_U \tr R(u,v) d\mu(u) p_\varepsilon \right) + \int_U rg(p_\varepsilon,u,v)d\mu(u).\]
We deduce that
\begin{align*}
rM_\varepsilon-r\varepsilon \ell(p_\varepsilon) &=  r(W_2(p_\varepsilon)-W_1(p_\varepsilon)) \\ 
&\leq \inf_{v \in V} \; \sup_{\mu \in \Delta(U)} \; \vec D W_2 \left(p_\varepsilon; \int_U \tr R(u,v) d\mu(u) p_\varepsilon \right) + \int_U rg(p_\varepsilon,u,v)d\mu(u) \\
& \quad -  \inf_{v \in V} \;\sup_{\mu \in \Delta(U)} \; \vec D W_1 \left(p_\varepsilon; \int_U \tr R(u,v) d\mu(u) p_\varepsilon \right) + \int_U r g(p_\varepsilon,u,v)d\mu(u) .
\end{align*}
Let $v_\varepsilon \in V$ such that
\begin{multline*}
\sup_{\mu \in \Delta(U)} \; \vec D W_1 \left(p_\varepsilon; \int_U \tr R(u,v_\varepsilon) d\mu(u) p_\varepsilon \right) + \int_U rg(p_\varepsilon,u,v_\varepsilon)d\mu(u) \\
\leq
\inf_{v \in V} \; \sup_{\mu \in \Delta(U)} \; \vec D W_1 \left(p_\varepsilon; \int_U \tr R(u,v) d\mu(u) p_\varepsilon \right) + \int_U rg(p_\varepsilon,u,v)d\mu(u)+\varepsilon.
\end{multline*}
Choose then $\mu_\varepsilon \in \Delta(U)$ such that
\begin{multline*}
\vec D W_2 \left(p_\varepsilon; \int_U \tr R(u,v_\varepsilon) d\mu_\varepsilon(u) p_\varepsilon \right) + \int_U rg(p_\varepsilon,u,v_\varepsilon)d\mu_\varepsilon(u) \\
 \geq \sup_{\mu \in \Delta(U)} \; \vec D W_2 \left(p_\varepsilon; \int_U \tr R(u,v_\varepsilon) d\mu(u) p_\varepsilon \right) + \int_U rg(p_\varepsilon,u,v_\varepsilon)d\mu(u)-\varepsilon.
\end{multline*}
We obtain
\begin{align*}
rM_\varepsilon-r\varepsilon \ell(p_\varepsilon) &\leq\  \vec D W_2 \left(p_\varepsilon; \int_U \tr R(u,v_\varepsilon) d\mu_\varepsilon(u) p_\varepsilon \right) + \int_U rg(p_\varepsilon,u,v_\varepsilon)d\mu_\varepsilon(u) \\
& \quad -  \vec D W_1 \left(p_\varepsilon; \int_U \tr R(u,v_\varepsilon) d\mu_\varepsilon(u) p_\varepsilon \right)- \int_U rg(p_\varepsilon,u,v_\varepsilon)d\mu_\varepsilon(u)+2\varepsilon \\
&= \vec D W_2 \left(p_\varepsilon; \int_U \tr R(u,v_\varepsilon) d\mu_\varepsilon(u) p_\varepsilon \right)- \vec D W_1 \left(p_\varepsilon; \int_U \tr R(u,v_\varepsilon) d\mu_\varepsilon(u) p_\varepsilon \right) +2\varepsilon.
\end{align*}
Define $z_\varepsilon = \int_U \tr R(u,v_\varepsilon) d\mu_\varepsilon(u) p_\varepsilon \in T_{\Delta(K)}(p_\varepsilon)$ and note that $|z_\varepsilon| \leq C'$ for some constant $C'$ since $R$ is bounded.
Choose $\varepsilon$ sufficiently small so that 
\[ rM_\varepsilon-r\varepsilon \ell(p_\varepsilon) - 2\varepsilon -\varepsilon C '>0.\]
The map $ W_2(p_\varepsilon+tz_\varepsilon ) - (W_1(p_\varepsilon+tz_\varepsilon )- \varepsilon \ell(p_\varepsilon+tz_\varepsilon ))$ admits a right-derivative at $t=0$ equal to 
\begin{align*}
 \vec D W_2 (p_\varepsilon; z_\varepsilon)- \vec D W_1 (p_\varepsilon; z_\varepsilon) + \varepsilon \langle \nabla \ell(p_\varepsilon), z_\varepsilon \rangle &\geq  \vec D W_2 (p_\varepsilon; z_\varepsilon)- \vec D W_1 (p_\varepsilon; z_\varepsilon) -\varepsilon C'\\ 
 &\geq rM_\varepsilon-r\varepsilon \ell(p_\varepsilon) - 2\varepsilon -\varepsilon C'  \\
 & >0.
\end{align*}
This inequality contradicts the definition of $p_\varepsilon$ which concludes the proof.  
\end{proof}

\section{Existence of the value}\label{section_DPP}

This section is devoted to the proof of Theorems \ref{th:viscosity} and \ref{th:variational}. The proof is divided in two parts: At first we prove that $W^-$ is a viscosity supersolution of \eqref{eq:HJB}, which implies that $W^-$ satisfies \eqref{sup_var} thanks to Proposition \ref{prop:equiv_super}. Then, as in Cardaliaguet \cite{cardadiff}, we prove that $W^+$ satisfies \eqref{sub_var} through the analysis of its concave conjugate, which may be interpreted as the value of a dual game as introduced by De Meyer \cite{demeyerdual}. Using Proposition \ref{prop:equiv_sub}, this implies that $W^+$ is a viscosity subsolution of \eqref{eq:HJB}. Thanks to Proposition \ref{prop:comparison}, we conclude that that $W^-=W^+$ and that $W$ is the unique Lipschitz viscosity solution of \eqref{eq:HJB} and the unique concave Lipschitz function satisfying \eqref{sup_var} and \eqref{sub_var}.

\subsection{Proof of the supersolution property}

In this subsection, we prove $W^-$ satisfies a super dynamic programming inequality in Proposition \ref{superDPP} and we deduce that $W^-$ is a viscosity supersolution of \eqref{eq:HJB} in Proposition \ref{prop:viscosup}.

Let $\Sigma^* \subset \Sigma$ be the set of pure strategies which do not depend on the trajectory $(X_t)_{t \geq 0}$.

\begin{proposition}\label{superDPP}
For all $\varepsilon>0$
\begin{equation}\label{supDPP}
W^{-}(p) \geq \sup_{\alpha \in \Sigma^*} \inf_{\beta \in \CT} \EE_p^{\alpha,\beta}[\int_0^\varepsilon re^{-rs}g(X_s,u^{\alpha,\beta}_s,v^{\alpha,\beta}_s)ds + e^{-r\varepsilon}W^{-}(\pi^{\alpha,\beta}_{\varepsilon})],
\end{equation}
where $\pi^{\alpha,\beta}_{\varepsilon}= e^{\int_0^\varepsilon \tr R(u^{\alpha,\beta}_s,v^{\alpha,\beta}_s)ds }p$.
\end{proposition}
\begin{proof}
Let $\delta>0$ and $\alpha_0\in \Sigma^*$ such that 
\begin{multline*}
\inf_{\beta \in \CT} \EE_p^{\alpha_0,\beta}[\int_0^\varepsilon re^{-rs}g(X_s,u^{\alpha_0,\beta}_s,v^{\alpha_0,\beta}_s)ds + e^{-r\varepsilon}W^{-}(\pi^{\alpha_0,\beta}_{\varepsilon})] \\
 \geq \sup_{\alpha \in \Sigma^*} \inf_{\beta \in \CT} \EE_p^{\alpha,\beta}[\int_0^\varepsilon re^{-rs}g(X_s,u^{\alpha,\beta}_s,v^{\alpha,\beta}_s)ds + e^{-r\varepsilon}W^{-}(\pi^{\alpha,\beta}_{\varepsilon})] - \delta
\end{multline*}
Let $(A_m)_{m=1,...,N}$ be a measurable partition of $\Delta(K)$ of mesh smaller than $\delta$ and for all $m=1,...,N$, let $p_m \in A_m$. 
For all $m=1,...,N$, let $\alpha_m \in\widehat{\Sigma}$ such that 
\[ \inf_{\beta \in \CT} J(p_m,\alpha_m,\beta) \geq W^-(p_m) - \delta. \]
\p
Define $\bar\alpha \in \widehat\Sigma$ with probability space $(M_{\bar\alpha},\CA_{\bar\alpha}, \lambda_{\bar\alpha})=(\prod_{m=1}^N M_{\alpha_m}, \bigotimes_{m=1}^N\CA_{\alpha_m},\bigotimes_{m=1}^N \lambda_{\alpha_m})$ by the formula:  $\forall (\xi_{\bar\alpha},\omega,v,t) \in M_{\bar\alpha} \times \Omega \times \CV \times [0,\infty)$,
\[ \bar\alpha(\xi_{\bar\alpha},\omega,v)_t= \indic_{[0,\varepsilon]}(t)\alpha_0(v)_t + \indic_{(\varepsilon,\infty)}(t) \sum_{m=1}^N \indic_{A_m}(\Pi^{\alpha_0}_\varepsilon(v)) \alpha_m(\xi_{\alpha_m}, \omega|_{[\varepsilon,\infty)}, v|_{[\varepsilon,\infty)})_{t-\varepsilon}.\]
where 
\[\Pi^{\alpha_0}_\varepsilon(v)= e^{\int_0^\varepsilon \tr R(\alpha_0(v)_s,v_s)ds }p.\]
$\bar\alpha$ is therefore a well-defined strategy in $\widehat{\Sigma}$.
Let $\beta \in \CT$. Note that by construction, we have:
\[\pi_{\varepsilon}^{\bar\alpha,\beta}= \Pi^{\alpha_0}_\varepsilon(v^{\alpha_0,\beta}).\]
Define $h_{\varepsilon}=u^{\bar\alpha,\beta}|_{[0,\varepsilon]}=u^{\alpha_0,\beta}|_{[0,\varepsilon]}$, and note that $h_\varepsilon$ and $\pi_{\varepsilon}^{\bar\alpha,\beta}$ do not depend on $(\omega,\xi_{\bar \alpha})$.
\p
Let $T'$ denote the grid of $\beta$. Thanks to Lemma \ref{lem:grids}, we may  assume that $\varepsilon=t'_n$ for some integer $n$.
For all $u \in \CU_\varepsilon$, define the continuation strategy $\beta^\varepsilon(u) \in \CT$ by 
\[ \forall u' \in \CU, \; \beta^{\varepsilon}(u)(u')_t=\sum_{p \geq n} \indic_{(t'_p,t'_{p+1}]}(t+\varepsilon) \beta((u\oplus_\varepsilon u'))_{t+\varepsilon},\]
where 
\[ \forall (u,u') \in \CU_\varepsilon \times \CU, \; (u\oplus_\varepsilon u')_t = \indic_{[0,\varepsilon]}(t)u_t +\indic_{(\varepsilon,+\infty)}(t)u'_{t-\varepsilon}.\]
Note first that by construction, we have the identity 
\begin{equation}\label{eq:identic_strat}
(u^{\bar\alpha,\beta}_{\varepsilon+s},v^{\bar\alpha,\beta}_{\varepsilon+s})(\xi_{\bar\alpha},\omega) = \sum_{m=1}^N \indic_{\pi^{\bar\alpha,\beta}_\varepsilon \in A_m} (u^{\alpha_m,\beta^\varepsilon(h_\varepsilon)}_s,v^{\alpha_m,\beta^\varepsilon(h_\varepsilon)}_s) (\xi_{\alpha_m},\omega|_{[\varepsilon,\infty)}).
\end{equation}
Applying Lemma \ref{girsanov}, a version of the conditional law of $(X_{\varepsilon+s})_{s \geq 0}$ given $(\xi_{\bar\alpha}, X|_{[0,\varepsilon]})$ is 
\[ \sum_{m=1}^N \indic_{\pi^{\bar\alpha,\beta}_\varepsilon \in A_m} \PP_{\delta_{X_\varepsilon}}^{\alpha_m(\xi_{\alpha_m}), \beta^\varepsilon(h_\varepsilon)}.\]
Using this fact together with \eqref{eq:identic_strat}, we have for all $\beta \in \CT$
\begin{align*}
J(p,\bar\alpha,\beta)&= \EE_p^{\bar\alpha,\beta}[\int_0^\varepsilon re^{-rs}g(X_s,u^{\bar\alpha,\beta}_s,v^{\bar\alpha,\beta}_s)ds+ e^{-r \varepsilon}\int_0^\infty re^{-rs}g(X_{\varepsilon+s},u^{\bar\alpha,\beta}_{\varepsilon+s},v^{\bar\alpha,\beta}_{\varepsilon+s})ds ] \\
&=\EE_p^{\bar\alpha,\beta}[\int_0^\varepsilon re^{-rs}g(X_s,u^{\bar\alpha,\beta}_s,v^{\bar\alpha,\beta}_s)ds+ e^{-r \varepsilon}\EE_p^{\bar\alpha,\beta}[\int_0^\infty re^{-rs}g(X_{\varepsilon+s},u^{\bar\alpha,\beta}_{\varepsilon+s},v^{\bar\alpha,\beta}_{\varepsilon+s})ds | \xi_{\bar\alpha},X|_{[0,\varepsilon]} ] ]\\
&=\EE_p^{\bar\alpha,\beta}[\int_0^\varepsilon re^{-rs}g(X_s,u^{\bar\alpha,\beta}_s,v^{\bar\alpha,\beta}_s)ds \\
&\qquad + e^{-r \varepsilon}\sum_m \indic_{\pi^{\bar\alpha,\beta}_\varepsilon \in A_m}\EE_{\delta_{X_\varepsilon}}^{\alpha^{m}(\xi_{\alpha^m}),\beta^\varepsilon(h_\varepsilon)}[\int_0^\infty re^{-rs}g(\bar X_{s},u^{\alpha^m,\beta^\varepsilon(h_\varepsilon)}_{s}(\bar X),v^{\alpha^m,\beta^\varepsilon(h_\varepsilon)}_{s}(\bar X))ds  ]],
\end{align*}
where the canonical process was denoted $\bar X$ in the last expectation to avoid confusions.
Recall that that the process 
$(e^{-\int_0^t \tr R(u^{\bar\alpha,\beta}_s,v^{\bar\alpha,\beta}_s)ds}\delta_{X_t})_{t \geq 0}$  is a $(\PP^{\bar\alpha(\xi_{\bar\alpha}),\beta}_p, \FF^X)$ martingale (see Lemma \ref{girsanov}), which implies:
\begin{equation}\label{eq:marginal}
\EE_p^{\bar\alpha(\xi_{\bar\alpha}),\beta} [ e^{-\int_0^\varepsilon \tr R(u^{\bar\alpha,\beta}_s,v^{\bar\alpha,\beta}_s)ds} \delta_{X_\varepsilon}]=p \;\Longrightarrow\; \EE_p^{\bar\alpha(\xi_{\bar\alpha}),\beta} [ \delta_{X_{\varepsilon}}] = \pi^{\bar\alpha,\beta}_{\varepsilon}.
\end{equation}
We deduce that
\begin{align*}
J(p,\bar\alpha,\beta)&= \EE_p^{\bar\alpha,\beta}[\int_0^\varepsilon re^{-rs}g(X_s,u^{\bar\alpha,\beta}_s,v^{\bar\alpha,\beta}_s)ds \\
&\qquad + e^{-r \varepsilon}\sum_m \indic_{\pi^{\bar\alpha,\beta}_\varepsilon \in A_m}\EE_{\pi^{\bar\alpha,\beta}_\varepsilon}^{\alpha^{m}(\xi_{\alpha^m}),\beta^\varepsilon(h_\varepsilon)}[\int_0^\infty re^{-rs}g(\bar X_{s},u^{\alpha^m,\beta^\varepsilon(h_\varepsilon)}_{s}(\bar X),v^{\alpha^m,\beta^\varepsilon(h_\varepsilon)}_{s}(\bar X))ds  ]]\\
&=\EE_p^{\bar\alpha,\beta}[\int_0^\varepsilon re^{-rs}g(X_s,u^{\bar\alpha,\beta}_s,v^{\bar\alpha,\beta}_s)ds+ e^{-r \varepsilon}\sum_m \indic_{\pi^{\bar\alpha,\beta}_\varepsilon \in A_m} J(\pi^{\bar\alpha,\beta}_\varepsilon, \alpha^{m},\beta^\varepsilon(h_\varepsilon)) ]\\
&\geq \EE_p^{\bar\alpha,\beta}[\int_0^\varepsilon re^{-rs}g(X_s,u^{\bar\alpha,\beta}_s,v^{\bar\alpha,\beta}_s)ds+ e^{-r \varepsilon}\sum_m \indic_{\pi^{\bar\alpha,\beta}_\varepsilon \in A_m} J(p_m, \alpha^{m},\beta^\varepsilon(h_\varepsilon)) ] - \sqrt{|K|}\delta  \\
&\geq \EE_p^{\bar\alpha,\beta}[\int_0^\varepsilon re^{-rs}g(X_s,u^{\bar\alpha,\beta}_s,v^{\bar\alpha,\beta}_s)ds+ e^{-r \varepsilon}\sum_m \indic_{\pi^{\bar\alpha,\beta}_\varepsilon \in A_m} W^-(p_m) ] - \sqrt{|K|}\delta- \delta \\
& \geq \EE_p^{\bar\alpha,\beta}[\int_0^\varepsilon re^{-rs}g(X_s,u^{\bar\alpha,\beta}_s,v^{\bar\alpha,\beta}_s)ds+ e^{-r \varepsilon} W^-(\pi^{\bar\alpha,\beta}_\varepsilon) ] - 2\sqrt{|K|}\delta- \delta 
\end{align*}
We conclude that 
\begin{align*}
W^+(p) &\geq \inf_{\beta \in \CT}  J(p,\bar\alpha,\beta) \\
& \geq \inf_{\beta \in \CT}\EE_p^{\bar\alpha,\beta}[\int_0^\varepsilon re^{-rs}g(X_s,u^{\bar\alpha,\beta}_s,v^{\bar\alpha,\beta}_s)ds+ e^{-r \varepsilon} W^-(\pi^{\bar\alpha,\beta}_\varepsilon) ] - 2\sqrt{|K|}\delta- \delta  \\
& \geq \sup_{\alpha \in \Sigma^*} \inf_{\beta \in \CT} \EE_p^{\alpha,\beta}[\int_0^\varepsilon re^{-rs}g(X_s,u^{\alpha,\beta}_s,v^{\alpha,\beta}_s)ds + e^{-rh}W^{-}(\pi^{\alpha,\beta}_{\varepsilon})] - 2\sqrt{|K|}\delta- 2\delta ,
\end{align*}
and the result follows by sending $\delta$ to zero.
\end{proof}

\begin{proposition}\label{prop:viscosup} 
$W^-$ is a viscosity supersolution of \eqref{eq:HJB}.
\end{proposition}
\begin{proof}
Assume that the property does not hold. Then there exist  $p \in \Delta(K)$ and $\phi$ a smooth test function such that $\phi \leq W^-$ on $\Delta(K)$, $\phi(p)=W^-(p)$ and   
\[r \phi(p) < H(p,\nabla\phi(p))=\sup_{u \in U} \; \inf_{v \in V} \; \langle \nabla \phi(p),\tr R(u,v) p \rangle + rg(p,u,v). \]
Therefore there exist $u_0$ and $\delta>0$ such that for all $v\in V$ 
\[r \phi(p) \leq   \langle \nabla\phi(p),\tr R(u_0,v) p \rangle + rg(p,u_0,v)  - \delta\]
Let $\alpha_0 \in \Sigma^*$ be the strategy which plays the constant control $u_0$ so that for all $\beta \in \CT$, $(u^{\alpha_0,\beta},v^{\alpha_0,\beta})=(u_0,\beta(u_0))$.
Applying \eqref{supDPP}, we have  
\begin{equation}
W^{-}(p) \geq \inf_{\beta \in \CT} \EE_p^{\alpha_0,\beta}[\int_0^\varepsilon re^{-rs}g(X_s,u_0,\beta(u_0)_s)ds + e^{-r\varepsilon}W^{-}(\pi_{\varepsilon}^{\alpha_0,\beta})] ,
\end{equation}
which implies 
\begin{equation}
\phi(p) \geq \inf_{\beta \in \CT}\EE_p^{\alpha_0,\beta}[\int_0^\varepsilon re^{-rs}g(X_s,u_0,\beta(u_0)_s)ds + e^{-r\varepsilon}\phi(\pi_{\varepsilon}^{\alpha_0,\beta_\varepsilon})],
\end{equation}
and thus
\[(1-e^{-r \varepsilon})\phi(p) \geq \inf_{\beta \in \CT} \EE_p^{u_0,\beta}[\int_0^\varepsilon re^{-rs}g(X_s,u_0,\beta(u_0)_s)ds + e^{-r\varepsilon}(\phi(\pi^{\alpha_0,\beta}_{\varepsilon})-\phi(p))]. \]
Let $ \pi^{\alpha_0,\beta}_s= e^{\int_0^s \tr R(u_0,\beta(u_0)_s)ds}p$ for all $s\in [0,\varepsilon]$. Since $\phi$ is smooth and $R$ is bounded, there exists a constant $C$ such that for all $\beta \in \CT$ 
\begin{align*}
\phi(\pi^{\alpha_0,\beta}_{\varepsilon})-\phi(p) &= \int_0^\varepsilon \langle \nabla \phi(\pi^{\alpha_0,\beta}_s), \tr R(u_0,\beta(u_0)_s) \pi^{\alpha_0,\beta}_s \rangle ds  \\
&\geq \int_0^\varepsilon \langle \nabla \phi(p), \tr R(u_0,\beta(u_0)_s) p \rangle ds - C\varepsilon^2 .\end{align*}
Lemma \ref{girsanov} implies that
\[ \EE_p^{\alpha_0,\beta}[e^{-\int_0^s \tr R(u_0,\beta(u_0)_s)ds} \delta_{X_s}] =p \; \Longrightarrow \EE_p^{\alpha_0,\beta}[ \delta_{X_s}] =\pi^{\alpha_0,\beta}_s.\] 
Using that $g$ is bounded and Lipschitz with respect to $p$, there exists a constant $C'$ such that
\begin{align*}
\EE_p^{\alpha_0,\beta}[\int_0^\varepsilon re^{-rs}g(X_s,u_0,\beta(u_0)_s)ds]&= 
\int_0^\varepsilon re^{-rs}g(\pi^{\alpha_0,\beta}_s,u_0,\beta(u_0)_s)ds \\
&\geq re^{-r \varepsilon}\int_0^\varepsilon g(p,u_0,\beta(u_0)_s)ds - C'\varepsilon^2
\end{align*}
We deduce that
\begin{align*}
(1-e^{-r \varepsilon})\phi(p) &\geq \inf_{\beta \in \CT}  e^{-r\varepsilon}(\int_0^\varepsilon \langle \nabla \phi(p), \tr R(u_0,\beta(u_0)_s) p \rangle ds +rg(p,u_0,\beta(u_0)_s)) - (C+C') \varepsilon^2 \\
&\geq e^{-r\varepsilon}\varepsilon(r\phi(p)+\delta) - (C+C') \varepsilon^2. 
\end{align*}
Dividing by $\varepsilon$ and sending $\varepsilon$ to zero, we obtain a contradiction and this concludes the proof.
\end{proof}

\subsection{Proof of the subsolution property}

This section is devoted to the proof that $W^+$ satisfies \eqref{sub_var}.
To this end, we consider the concave conjugate defined by
\[ \forall x \in \RR^K, \;  W^{+,*}(x)=\inf_{p \in \Delta(K)} \langle x,p \rangle - W^+(p).\]
In Proposition \ref{prop:dualDPP}, we will prove that $W^{+,*}$ satisfies a dynamic programming inequality and in Proposition \ref{prop:dualvisco}, we will prove that this implies that $W^{+,*}$ is a viscosity supersolution of the following dual equation for $x\in \RR^K$:
\begin{equation} \label{eq:dualHJB}
r f(x)+H(\nabla f(x),x) -r \langle \nabla f(x),x \rangle       \geq 0.
\end{equation}
Note that for the above equation to be well-defined, the definition of $H$ has to be extended to $\RR^K\times \RR^K$, for example by letting
\[ \forall (p,x) \in \RR^{K}\times \RR^{K}, \;H(p,x)= \inf_{v \in V} \;\sup_{u \in U} \; \langle \tr R(u,v) p, z\rangle + rg(p,u,v),\] 
where $g(p,u,v)=\sum_{k \in K}p_k g(k,u,v)$. Note that using the same arguments as in Lemma \ref{lem:properties}, $H$ is locally Lipschitz with respect to both variables.
\p
Let us recall the precise definition of a viscosity supersolution of \eqref{eq:dualHJB}.
\begin{definition}
A function $f:\RR^K\mapsto\RR$ is called a supersolution of \eqref{eq:dualHJB} if it is lower semi-continuous and satisfies: for any smooth test function $\varphi:\RR^K\mapsto\RR$ and $x\in\RR^K$ such that $\varphi-f$ has a global maximum at $x$, we have
\[ r \varphi(x)+H(\nabla \varphi (x),x) -r \langle \nabla \varphi(x),x \rangle       \geq 0.\]
\end{definition}
In Proposition \ref{prop:dualprimal}, we will deduce that $W^+$ satisfies \eqref{sub_var} from the fact that $W^{+,*}$ is a viscosity supersolution of \eqref{eq:dualHJB}. 
\p
We start with an alternative representation for $W^{+,*}$.
\begin{lemma}\label{lem:alternat}
\begin{align*} 
W^{+,*}(x)&=\sup_{\beta \in \widehat{\CT}} \;\inf_{\alpha \in {\Sigma}}\; \inf_{p \in \Delta(K)} \;\langle x,p \rangle -J(p,\alpha,\beta)
\end{align*}
\end{lemma}
\begin{proof}
We consider the map 
\[ \Theta:(p,\beta) \in \Delta(K) \times \widehat\CT \rightarrow \inf_{\alpha \in {\Sigma}}\;\langle x,p \rangle -J(p,\alpha,\beta) .\]
In order to apply Fan's minmax theorem, we first verify that $\Theta$ is affine with respect to $p$ (and thus continuous) on the compact convex set $\Delta(K)$ and 
 that $\Theta$ is concave-like with respect to $\beta$ on the set $\widehat{\CT}$. 
For the first part, recall that for all $\beta \in \widehat{\CT}$
\begin{align*}
\Theta(p,\beta)&=\langle x,p \rangle -\sup_{\alpha \in {\Sigma}}\;J(p,\alpha,\beta) \\
&=\langle x,p \rangle -\sup_{\alpha \in {\Sigma}}\;\sum_{k \in K}p_k J(\delta_k,\alpha,\beta)\\
&=\langle x,p \rangle -\;\sum_{k \in K}p_k \sup_{\alpha \in {\Sigma}}J(\delta_k,\alpha,\beta),
\end{align*}
where the second equality was proved in Lemma \ref{lem:properties}. 
Let $\beta_1,\beta_2 \in \widehat\CT$ and $\lambda \in [0,1]$. Define a strategy  
 $\bar\beta \in \widehat{\CT}$ having for probability space $(M_{\bar\beta},\CA_{\bar\beta},\lambda_{\bar\beta})$ defined by  
\[M_{\bar\beta}=[0,1]\times M_{\beta_1} \times M_{\beta_2}, \; \CA_{\bar\beta}=\CB([0,1])\otimes \CA_{\beta_1}\otimes \CA_{\beta_2}, \; \lambda_{\bar\beta}=\Leb\otimes \lambda_{\beta_2}\otimes \lambda_{\beta_2},\] 
where $\Leb$ denotes the Lebesgue measure on $[0,1]$.
A typical element of $M_\beta$ will be denoted $(\zeta,\xi_{\beta_1},\xi_{\beta_2})$.
The strategy $\bar\beta$ is defined by 
\[\bar\beta(\zeta,\xi_{\beta_1},\xi_{\beta_2})= \indic_{\zeta\leq \lambda} \beta_1(\xi_{\beta_1})+ \indic_{\zeta >\lambda}\beta_1(\xi_{\beta_1}).\]
With this definition, integrating with respect to $\zeta$, we have for all $\alpha \in \Sigma$,
\[ J(p,\alpha,\bar\beta)= \lambda J(p,\alpha,\beta_1) +(1-\lambda)J(p,\alpha,\beta_2).\]
It follows that
\begin{align*}
\Theta(p,\bar \beta)&= \inf_{\alpha \in {\Sigma}}\; \left\{ \lambda (\langle x,p \rangle -J(p,\alpha,\beta_1)) +(1-\lambda)(\langle x,p \rangle -J(p,\alpha,\beta_2))\right\} \\
&\geq \inf_{\alpha \in {\Sigma}}\; \lambda (\langle x,p \rangle -J(p,\alpha,\beta_1)) + \inf_{\alpha \in {\Sigma}}\; (1-\lambda)(\langle x,p \rangle -J(p,\alpha,\beta_2))\\
&=\lambda \Theta(p,\beta_1)+ (1-\lambda)\Theta(p,\beta_2),
\end{align*}
which concludes the proof of the concave-like property.
\p
Fan's minmax theorem (see \cite{fan}) implies
\begin{align*} 
W^{+,*}(x)&=\inf_{p \in \Delta(K)} \langle x,p \rangle - W^+(p) \\
&=\inf_{p \in \Delta(K)}\; \sup_{\beta \in \widehat{\CT}} \;\inf_{\alpha \in {\Sigma}}\;\langle x,p \rangle -J(p,\alpha,\beta) \\
&=\sup_{\beta \in \widehat{\CT}} \;\inf_{\alpha \in {\Sigma}}\; \inf_{p \in \Delta(K)} \;\langle x,p \rangle -J(p,\alpha_r,\beta_r)
\end{align*}
\end{proof}

\begin{proposition}\label{prop:dualDPP}
\begin{align*}
W^{+,*}(x) &\geq  \sup_{\beta\in\CT} \;   \inf_{u \in \CU}\;e^{-r\varepsilon} W^{+,*}(Z^{\beta}_\varepsilon(u))
\end{align*}
where 
\[Z^{\beta}_\varepsilon(u):= e^{r\varepsilon} \left( Y^{\beta}_\varepsilon(u) - \int_0^\varepsilon re^{-rt}G^{\beta}_{t,\varepsilon}(u) dt   \right),\]
\[ Y^\beta_\varepsilon(u) =e^{-\int_0^\varepsilon R(u_s,\beta(u)_s)ds} \; x\]
and for all $t\in [0,\varepsilon]$
\[ G^\beta_{t,\varepsilon} (u) = e^{-\int_t^\varepsilon R(u_s,\beta(u)_s)ds} \; \bar g(u_t,\beta(u)_t)\]
with for all $(a,b)\in U\times V$, $\bar g(a,b)= (g(k,a,b))_{k \in K} \in \RR^K$.
\end{proposition}

\begin{proof}
Let $\delta>0$ and $\beta_0 \in \CT$ such that 
\begin{align} \label{eq:dualDPP0}
\inf_{u \in \CU}\;&  e^{-r\varepsilon} W^{+,*}(Z^{\beta_0}_\varepsilon(u)) \geq \sup_{\beta\in\CT} \;\inf_{u \in \CU}\;  e^{-r\varepsilon} W^{+,*}(Z^{\beta}_\varepsilon(u)) - \delta
\end{align}
Note that since $R$ and $g$ are bounded, there exists a compact set $C_\varepsilon$ such that 
\[ \forall u \in \CU,\; Z^{\beta_0}_\varepsilon(u) \in C_\varepsilon.\] 
Let $(A_m)_{m=1,...,N}$ be a measurable partition of $C_\varepsilon$ with mesh smaller than $\delta$ and for all $m=1,...,N$, let $z_m \in A_m$.
For all $m=1,...,N$, let $\beta_m \in \widehat{\CT}$ such that
\[ \inf_{\alpha \in \Sigma}\inf_{p \in \Delta(K)} \langle z_m , p \rangle - J(p,\alpha,\beta_m) \geq W^+(z_m) - \delta .\]
We now construct a strategy $\bar\beta$ with probability space 
\[(M_{\bar\beta},\CA_{\bar\beta}, \lambda_{\bar\beta})=(\prod_{m=1}^N M_{\beta_m}, \bigotimes_{m=1}^N\CA_{\beta_m},\bigotimes_{m=1}^N \lambda_{\beta_m})\] 
defined by the formula:  $\forall (\xi_{\bar\beta},u,t) \in M_{\bar\beta}  \times \CU \times [0,\infty)$,
\[ \bar\beta(\xi_{\bar\beta},u)_t= \indic_{[0,\varepsilon]}(t)\beta_0(u)_t + \indic_{(\varepsilon,\infty)}(t) \sum_{m=1}^N \indic_{A_m}(Z^{\beta_0}_\varepsilon(u)) \beta_m(\xi_{\beta_m},  u|_{[\varepsilon,\infty)})_{t-\varepsilon}.\]
$\bar\beta$ is therefore a well-defined strategy in $\widehat{\CT}$.
Let us fix a strategy $\alpha \in \Sigma$.
Note that by construction, we have $(u^{\alpha,\beta_0},v^{\alpha,\beta_0})|_{[0,\varepsilon]}=(u^{\alpha,\bar\beta},v^{\alpha,\bar\beta})|_{[0,\varepsilon]}$ and therefore
\[ \forall t \in [0,\varepsilon],\; G^{\beta_0}_{t,\varepsilon}(u^{\alpha,\beta_0})= e^{-\int_t^\varepsilon R(u^{\alpha,\bar\beta}_s,v^{\alpha,\bar\beta}_s)ds} \; \bar g(u^{\alpha,\bar\beta}_t,v^{\alpha,\bar\beta}_t), \]
\[ \; Y^{\beta_0}_\varepsilon(u^{\alpha,\beta_0})=e^{-\int_0^\varepsilon R(u^{\alpha,\bar\beta}_s,v^{\alpha,\bar\beta}_s)ds} \; x, \]
and that all these variables depend only on $u^{\alpha,\beta_0}$ through $u^{\alpha,\beta_0}|_{[0,\varepsilon]}$ and thus are measurable functions of $\omega|_{[0,\varepsilon]}$.
Recall (Lemma \ref{girsanov}) that the process 
\[M^{\alpha,\bar\beta}_t = e^{-\int_0^t \tr R(u^{\alpha,\bar\beta}_s,v^{\alpha,\bar\beta}_s)ds}\delta_{X_t},\]
is a $(\PP_p^{\alpha,\bar\beta(\xi_{\bar\beta})}, \FF^X)$ martingale, which implies
\begin{align} \label{eq:dualDPP1}
\langle x,p\rangle&=\EE_p^{\alpha,\bar\beta(\xi_{\bar\beta})}[ \langle x,e^{-\int_0^\varepsilon \tr R(u^{\alpha,\bar\beta}_s,v^{\alpha,\bar\beta}_s)ds}\delta_{X_\varepsilon} \rangle ] = \EE_p^{\alpha,\bar\beta(\xi_{\bar\beta})}[\langle Y^{\beta_0}_\varepsilon(u^{\alpha,\beta_0}),\delta_{X_\varepsilon} \rangle] .
\end{align}
Similarly, we have 
\begin{align}
\EE_p^{\alpha,\bar\beta(\xi_{\bar\beta})}&[ \int_0^\varepsilon re^{-rt}g(X_t,u^{\alpha,\bar\beta}_t,v^{\alpha,\bar\beta}_t)dt] \notag \\
&= \EE_p^{\alpha,\bar\beta(\xi_{\bar\beta})}[ \int_0^\varepsilon re^{-rt} \left\langle \delta_{X_t} ,\bar g(u^{\alpha,\bar\beta}_t,v^{\alpha,\bar\beta}_t) \right\rangle dt] \notag\\
&=\EE_p^{\alpha,\bar\beta(\xi_{\bar\beta})}[ \int_0^\varepsilon re^{-rt} \left\langle \EE_p^{\alpha,\bar\beta(\xi_{\bar\beta})}[e^{-\int_{t}^\varepsilon \tr R(u^{\alpha,\bar\beta}_s,v^{\alpha,\bar\beta}_s)ds} \delta_{X_\varepsilon} |\CF^X_t] \, , \,\bar g(u^{\alpha,\bar\beta}_t,v^{\alpha,\bar\beta}_t) \right\rangle dt] \notag \\
&=\EE_p^{\alpha,\bar\beta(\xi_{\bar\beta})}[ \int_0^\varepsilon re^{-rt} \left\langle e^{-\int_{t}^\varepsilon \tr R(u^{\alpha,\bar\beta}_s,v^{\alpha,\bar\beta}_s)ds} \delta_{X_\varepsilon} \, , \,\bar g(u^{\alpha,\bar\beta}_t,v^{\alpha,\bar\beta}_t) \right\rangle dt] \notag \\
&=\EE_p^{\alpha,\bar\beta(\xi_{\bar\beta})}[ \int_0^\varepsilon re^{-rt} \left\langle  \delta_{X_\varepsilon} \, , \, G^{\beta_0}_{t,\varepsilon}(u^{\alpha,\beta_0}) \right\rangle dt]\notag \\
&=\EE_p^{\alpha,\bar\beta(\xi_{\bar\beta})}[\left\langle \delta_{X_\varepsilon} \, , \,  \int_0^\varepsilon re^{-rt}   G^{\beta_0}_{t,\varepsilon}(u^{\alpha,\beta_0}) \right\rangle dt]. \label{eq:dualDPP2}
\end{align}
Define $h_{\varepsilon}=(\omega,v^{\alpha,\bar\beta})|_{[0,\varepsilon]}=(\omega,v^{\alpha,\beta^0})|_{[0,\varepsilon]}$, and note that $h_\varepsilon$ is a measurable function of $\omega|_{[0,\varepsilon]}$.
Let $T'$ denote the grid of $\alpha$. We may  assume that $\varepsilon=t'_n$ for some integer $n$.
For all $(\omega,v) \in \CU_\varepsilon\times \Omega_\varepsilon$, define the continuation strategy $\alpha^\varepsilon(\omega,v) \in \Sigma$  by 
\[ \forall v' \in \CV,\forall \omega'\in \Omega \; \alpha^{\varepsilon}(\omega,v)(\omega',v')_t=\sum_{p \geq n} \indic_{(t'_p,t'_{p+1}]}(t+\varepsilon) \alpha(\omega\oplus_{\varepsilon-}\omega', v\oplus_\varepsilon v')_{t+\varepsilon},\]
where 
\[ \forall (v,v') \in \CU_\varepsilon \times \CU, \; (v\oplus_\varepsilon v')_t = \indic_{[0,\varepsilon]}(t)v_t +\indic_{(\varepsilon,+\infty)}(t)v'_{t-\varepsilon},\]
and
\[ \forall (\omega,\omega') \in \Omega_\varepsilon \times \Omega, \; (\omega\oplus_{\varepsilon-} \omega')_t = \indic_{[0,\varepsilon)}(t)\omega(t) +\indic_{[\varepsilon,+\infty)}(t)\omega'(t-\varepsilon).\]
\p
Note first that by construction, we have the identity 
\begin{multline}\label{eq:identic_strat_2}
(u^{\alpha,\bar\beta}_{\varepsilon+s},v^{\alpha,\bar\beta}_{\varepsilon+s})(\xi_{\bar\beta},\omega) 
= \sum_{m=1}^N \indic_{A_m}(Z^{\beta_0}_\varepsilon(u^{\alpha,\beta_0})) (u^{\alpha^\varepsilon(h_\varepsilon),\beta^m}_s,v^{\alpha^\varepsilon(h_\varepsilon),\beta^m}_s) (\xi_{\beta^m},\omega|_{[\varepsilon,\infty)}).
\end{multline}
According to Lemma \ref{girsanov}, the map 
\[ (\xi_{\bar\beta},\omega|_{[0,\varepsilon]}) \rightarrow \Phi(\xi_{\bar\beta},\omega|_{[0,\varepsilon]})=\sum_{m=1}^N \indic_{A_m}(Z^{\beta_0}_\varepsilon(u^{\alpha,\beta_0})) \PP_{\delta_{X_\varepsilon}}^{\alpha^\varepsilon(h_\varepsilon), \beta_m(\xi_{\beta_m})},\]
is a version of the conditional distribution of $(X_{\varepsilon+s})_{s \geq 0}$ given $(\xi_{\bar\beta},X|_{[0,\varepsilon]})$.
Using these results, we have
\begin{align*}
J(p,\alpha,\bar\beta)&= \EE_p^{\alpha,\bar\beta}[\int_0^\varepsilon re^{-rt}g(X_t,u^{\alpha,\bar\beta}_t,v^{\alpha,\bar\beta}_t)dt + e^{-r \varepsilon} \int_0^\infty re^{-rt}g(X_{t+\varepsilon},u^{\alpha,\bar\beta}_{t+\varepsilon},v^{\alpha,\bar\beta}_{t+\varepsilon})dt]\\
&=\EE_p^{\alpha,\bar\beta}[\int_0^\varepsilon re^{-rt}g(X_t,u^{\alpha,\bar\beta}_t,v^{\alpha,\bar\beta}_t)dt + e^{-r \varepsilon} \EE_p^{\alpha,\bar\beta}[\int_0^\infty re^{-rt}g(X_{t+\varepsilon},u^{\alpha,\bar\beta}_{t+\varepsilon},v^{\alpha,\bar\beta}_{t+\varepsilon})dt | \xi_{\bar\beta},  X|_{[0,\varepsilon]}]]\\
&=\EE_p^{\alpha,\bar\beta}[\int_0^\varepsilon re^{-rt}g(X_t,u^{\alpha,\bar\beta}_t,v^{\alpha,\bar\beta}_t)dt \\
& + e^{-r \varepsilon}  \sum_{m=1}^N \indic_{A_m}(Z^{\beta_0}_\varepsilon(u^{\alpha,\beta_0})) \EE_{\delta_{X_\varepsilon}}^{\alpha^\varepsilon(h_\varepsilon), \beta_m} [\int_0^\infty re^{-rt}g(\bar X_t,u^{\alpha^{\varepsilon}(h_\varepsilon),\beta_m}_{t}(\bar X),v^{\alpha^{\varepsilon}(h_\varepsilon),\beta_m}_{t}(\bar X))dt ] ]\\
&=\EE_p^{\alpha,\bar\beta}[\int_0^\varepsilon re^{-rt}g(X_t,u^{\alpha,\bar\beta}_t,v^{\alpha,\bar\beta}_t)dt + e^{-r \varepsilon} \sum_{m=1}^N \indic_{A_m}(Z^{\beta_0}_\varepsilon(u^{\alpha,\beta_0})) J(\delta_{X_\varepsilon},\alpha^\varepsilon(h_\varepsilon),\beta_m) ],
\end{align*}
where we used the notation $\bar X$ for the canonical process on $\Omega$ in the fourth line to avoid potential confusions.   
Using now \eqref{eq:dualDPP0},\eqref{eq:dualDPP1} and \eqref{eq:dualDPP2}, we have (with the shorter notation $Z^{\beta_0}_\varepsilon=Z^{\beta_0}_\varepsilon(u^{\alpha,\beta_0})$)
\begin{align*}
W^{+,*}(x)&\geq \;\inf_{\alpha\in\Sigma}\; \inf_{p \in \Delta(K)}  e^{-r\varepsilon}\EE_p^{\alpha,\bar\beta}[ \langle Z^{\beta_0}_\varepsilon \;,\;\delta_{X_\varepsilon} \rangle - \sum_{m=1}^N \indic_{A_m}(Z^{\beta_0}_\varepsilon) J(\delta_{X_\varepsilon}, \alpha^\varepsilon(h_\varepsilon), \beta_m) ] \\
&\geq \;\inf_{\alpha\in\Sigma}\; \inf_{p \in \Delta(K)}  e^{-r\varepsilon}\EE_p^{\alpha,\bar\beta}[  \sum_{m=1}^N \indic_{A_m}(Z^{\beta_0}_\varepsilon) \{\langle Z^{\beta_0}_\varepsilon\;,\;\delta_{X_\varepsilon} \rangle -  J(\delta_{X_\varepsilon}, \alpha^\varepsilon(h_\varepsilon), \beta_m) \} ] \\
&\geq \;\inf_{\alpha\in\Sigma}\; \inf_{p \in \Delta(K)}  e^{-r\varepsilon}\EE_p^{\alpha,\bar\beta}[  \sum_{m=1}^N \indic_{A_m}(Z^{\beta_0}_\varepsilon) \{ \langle z_m \;,\;\delta_{X_\varepsilon}\rangle -  J(\delta_{X_\varepsilon}, \alpha^\varepsilon(h_\varepsilon), \beta_m) \} ] - \delta \\
&\geq \;\inf_{\alpha\in\Sigma}\; \inf_{p \in \Delta(K)}  e^{-r\varepsilon}\EE_p^{\alpha,\bar\beta}[  \sum_{m=1}^N \indic_{A_m}(Z^{\beta_0}_\varepsilon) \{ W^{+,*}(z_m) - \delta \} ] - \delta \\
&\geq \;\inf_{\alpha\in\Sigma}\; \inf_{p \in \Delta(K)}  e^{-r\varepsilon}\EE_p^{\alpha,\beta_0}[   W^{+,*}(Z^{\beta_0}_\varepsilon)  ]- \sqrt{K} \delta - 2\delta \\
&\geq \sup_{\beta \in {\CT}}\;\inf_{\alpha\in\Sigma}\; \inf_{p \in \Delta(K)}  e^{-r\varepsilon}\EE_p^{\alpha,\beta}[  W^{+,*}(Z^{\beta}_\varepsilon(u^{\alpha,\beta}))  ]- \sqrt{K} \delta - 3\delta 
\end{align*}
By sending $\delta $ to zero, we deduce that
\begin{align*}
W^{+,*}(x)&\geq \sup_{\beta \in {\CT}}\;\inf_{\alpha\in\Sigma}\; \inf_{p \in \Delta(K)}  e^{-r\varepsilon}\EE_p^{\alpha,\beta}[   W^{+,*}(Z^{\beta}_\varepsilon(u^{\alpha,\beta}))  ] 
\end{align*}
Note finally that the quantity inside the expectation $\EE_p^{\alpha,\beta}[   W^{+,*}(Z^{\beta}_\varepsilon(u^{\alpha,\beta}))  ]$ depends on $X$ only through the process $u^{\alpha,\beta}$, and therefore
\[ \EE_p^{\alpha,\beta}[   W^{+,*}(Z^{\beta}_\varepsilon(u^{\alpha,\beta}))  ] \geq \inf_{u \in \CU}  W^{+,*}(Z^{\beta}_\varepsilon(u)). \]
We deduce that
\begin{align*}
W^{+,*}(x)&\geq \sup_{\beta \in {\CT}}  \inf_{u \in \CU} e^{-r\varepsilon} W^{+,*}(Z^{\beta}_\varepsilon(u))  , 
\end{align*}
which concludes the proof. 
\end{proof}

\begin{proposition}\label{prop:dualvisco}
The map $W^{+,*}$ is a viscosity supersolution of \eqref{eq:dualHJB}.
\end{proposition}
\begin{proof}
Assume that the property does not hold. Then there exists a smooth test function $\phi$ and $x\in \RR^K$ such that $\phi \leq W^{+,*}$ on $\RR^K$ with equality at $x$ and
\[ r \phi(x)+H(\nabla \phi(x),x) -r \langle \nabla \phi (x),x \rangle       <  0.\]
This implies that there exists $v_0 \in V$ and $\delta>0$ such that for all $u\in U$
\[ r \phi(x)+ \langle  R(u,v_0) x + r \bar g(u,v_0), \nabla \phi(x) \rangle  -r \langle \nabla \phi (x),x \rangle       \leq  - \delta.\]
Thanks to Proposition \ref{prop:dualDPP}, for all $\varepsilon>0$, we have
\begin{align*}
\phi(x) =W^{+,*}(x)\geq \sup_{\beta \in {\CT}} \inf_{u \in \CU} e^{-r\varepsilon} W^{+,*}(Z^{\beta}_\varepsilon(u))\geq \sup_{\beta \in {\CT}} \inf_{u \in \CU} e^{-r\varepsilon} \phi(Z^{\beta}_\varepsilon(u)) . 
\end{align*}
Considering the pure strategy $\beta_0$ which plays the constant control $v_0$, we have for all $\varepsilon>0$
\begin{align*}
\phi(x)\geq \inf_{u \in \CU} e^{-r\varepsilon} \phi\left( Z^{\beta_0}_{\varepsilon}(u)  \right). 
\end{align*}
Define the absolutely continuous map (which depends on $v_0$, $u$ and $\varepsilon$)
\[t \in [0,\varepsilon]  \rightarrow \tilde Z_t =e^{rt} \left( e^{-\int_0^t R(u_s,v_0)ds} \; x -\int_0^t re^{-rs}e^{-\int_s^\varepsilon R(u_r,v_0)dr} \; \bar g(u_s,v_0) ds \right) \]
so that $Z^{\beta_0}_\varepsilon(u)=\tilde Z_\varepsilon$ and therefore we obtain by applying the chain rule formula:
\begin{multline*}
\phi(Z^{\beta_0}_\varepsilon(u))- \phi(x) = \\
\int_0^{\varepsilon} \left\langle \nabla\phi(\tilde Z_t(u)) \, , \, r\tilde Z_t + e^{rt}\left ( -R(u_t,v_0)e^{-\int_0^t R(u_s,v_0)ds} \; x -re^{-rt}e^{-\int_t^\varepsilon R(u_s,v_0)ds} \; \bar g(u_t,v_0)   \right) \right\rangle  dt
\end{multline*}
Since $R$, $\bar g$ are bounded and $\phi$ is smooth, there exists a constant $C$ such that
\[ \phi(Z_\varepsilon(u)) \geq \phi(x) + \int_0^{\varepsilon} \left\langle \nabla\phi(x) \, , \, \left ( rx -R(u_t,v_0)x -r \bar g(u_t,v_0) \right) \right\rangle  dt - C\varepsilon^2.\]
We obtain
\[ (1-e^{-r\varepsilon})\phi(x)  \geq \varepsilon (r\phi(x) + \delta) - C\varepsilon ^2\]
Dividing the above inequality by $\varepsilon$ and sending $\varepsilon$ to zero leads to a contradiction, which concludes the proof.\end{proof}

\begin{proposition}\label{prop:dualprimal}
$W^+$ satisfies \eqref{sub_var}.
\end{proposition}
\begin{proof}
First recall that $W^{+,*}$ is concave and Lipschitz since the domain of $W^+$ is bounded.
Let $p \in Exp(W^+)$ and $x \in\partial^+ W^+(p)$ such that
\[ \argmin_{p' \in \Delta(K)} \langle x, p' \rangle - W^+(p') = \{p\}.\]
Let $\hat W^+ : \RR^K \rightarrow \RR$ denote the Moreau-Yosida regularization of $W^+$ defined by 
\begin{equation*}
 \forall y\in \RR^K, \; \hat W^+(y)= \sup_{p' \in \Delta(K)} f(p') - M|y-p'|,
\end{equation*}
for some constant $M$ larger than the Lipschitz constant of $W^+$. It is well-known that $\hat{W}^+$ is concave and $M$-Lipschitz on $\RR^K$ and coincides with $W^+$ on $\Delta(K)$ so that
\begin{equation}\label{eq:direct_coincides_2}
 \forall p' \in \Delta(K), \forall z \in T_{\Delta(K)}(p'), \; \vec D W^+=(p';z)= \vec D \hat W^+(p';z)= \min_{x \in \partial^+ \hat W^+(p')} \langle x,z \rangle,
 \end{equation}  
where  $\partial^+ \hat W^+(p')$ is a compact convex subset of $\partial^+W^+(p')$.
\p
Let $y \in \partial^+ \hat W^+(p)$ and note that for all $\lambda \in (0,1]$, we have  
\[ \argmin_{p' \in \Delta(K)} \;\langle y_\lambda, p' \rangle - W^+(p') = \{p\}.\]
with $y_\lambda= \lambda x +(1-\lambda) y$.
Since 
\[ \partial^{+}W^{+,*}(y_\lambda)=\argmin_{p' \in \Delta(K)} \; \langle y_\lambda, p' \rangle - W^+(p') = \{p\},\]
we deduce that $W^{+,*}$ is differentiable at $y_\lambda$ and that $\nabla W^{+,*}(y_\lambda)=p$. Proposition \ref{prop:dualvisco} implies that for all $\lambda \in (0,1]$
\[ r W^+(p)= r (\langle p,y\lambda \rangle -W^{+,*}(y_\lambda)) \leq H(p, y_\lambda).\]
By sending $\lambda$ to zero, we obtain
\[ r W^+(p) \leq H(p,y),\]
and the proof by taking the minimum over all $y\in \partial^+ \hat W^+(p)$ and applying Lemma \ref{lem:minmaxH} and \eqref{eq:direct_coincides_2}.
\end{proof}




\end{document}